\documentclass[reqno,11pt]{amsart}
\usepackage{amsmath, latexsym, amsfonts, amssymb, amsthm, amscd}
\usepackage{graphics,epsf,psfrag}
\usepackage{graphics,epsf,psfrag}
 \usepackage[utf8]{inputenc}
\usepackage{stmaryrd}
\usepackage{mathabx}
\usepackage{hyperref}
\usepackage{xcolor}

\renewcommand{\bar}{\overline}

\newcommand{\lint}{\llbracket}
\newcommand{\rint}{\rrbracket}

\numberwithin{equation}{section}

\newtheorem{theorema}{Theorem}

\newtheorem{theorem}{Theorem}[section]
\newtheorem{lemma}[theorem]{Lemma}
\newtheorem{proposition}[theorem]{Proposition}

\newtheorem{rem}[theorem]{Remark}

\newcommand{\dd}{\mathrm{d}}
\newcommand{\ind}{\mathbf{1}}

\renewcommand{\tilde}{\widetilde}
\renewcommand{\hat}{\widehat}
\newcommand{\cc}{\complement}

\newcommand{\cA}{{\ensuremath{\mathcal A}} }

\newcommand{\cF}{{\ensuremath{\mathcal F}} }

\newcommand{\cP}{{\ensuremath{\mathcal P}} }

\newcommand{\cC}{{\ensuremath{\mathcal C}} }
\newcommand{\cN}{{\ensuremath{\mathcal N}} }

\newcommand{\cD}{{\ensuremath{\mathcal D}} }

\newcommand{\cZ}{{\ensuremath{\mathcal Z}} }

\newcommand{\cK}{{\ensuremath{\mathcal K}} }
\newcommand{\cM}{{\ensuremath{\mathcal M}} }

\newcommand{\bP}{{\ensuremath{\mathbf P}} }

\newcommand{\bE}{{\ensuremath{\mathbf E}} }

\DeclareMathSymbol{\leqslant}{\mathalpha}{AMSa}{"36} % nicer `smaller or equal'
\DeclareMathSymbol{\geqslant}{\mathalpha}{AMSa}{"3E} % nicer `larger or equal'
\DeclareMathSymbol{\eset}{\mathalpha}{AMSb}{"3F}     % nicer `emptyset'
\newcommand{\maxtwo}[2]{\max_{\substack{#1 \\ #2}}} % max with 2 lines
\newcommand{\limtwo}[2]{\lim_{\substack{#1 \\ #2}}}     % \lim with 2 lines
\newcommand{\bbC}{{\ensuremath{\mathbb C}} }

\newcommand{\bbE}{{\ensuremath{\mathbb E}} }

\newcommand{\bbL}{{\ensuremath{\mathbb L}} }

\newcommand{\bbN}{{\ensuremath{\mathbb N}} }

\newcommand{\bbP}{{\ensuremath{\mathbb P}} }

\newcommand{\bbR}{{\ensuremath{\mathbb R}} }

\newcommand{\gep}{\varepsilon}       % \ge already exists...

\newcommand{\gG}{\Gamma}

\newcommand{\gD}{\Delta}

\newcommand{\go}{\omega}

\newcommand{\gl}{\lambda}

\makeatletter
\def\captionfont@{\footnotesize}
\def\captionheadfont@{\scshape}

\long\def\@makecaption#1#2{%
  \vspace{2mm}
  \setbox\@tempboxa\vbox{\color@setgroup
    \advance\hsize-6pc\noindent
    \captionfont@\captionheadfont@#1\@xp\@ifnotempty\@xp
        {\@cdr#2\@nil}{.\captionfont@\upshape\enspace#2}%
    \unskip\kern-6pc\par
    \global\setbox\@ne\lastbox\color@endgroup}%
  \ifhbox\@ne % the normal case
    \setbox\@ne\hbox{\unhbox\@ne\unskip\unskip\unpenalty\unkern}%
  \fi
  \ifdim\wd\@tempboxa=\z@ % this means caption will fit on one line
    \setbox\@ne\hbox to\columnwidth{\hss\kern-6pc\box\@ne\hss}%
  \else % tempboxa contained more than one line
    \setbox\@ne\vbox{\unvbox\@tempboxa\parskip\z@skip
        \noindent\unhbox\@ne\advance\hsize-6pc\par}%
\fi
  \ifnum\@tempcnta<64 % if the float IS a figure...
    \addvspace\abovecaptionskip
    \moveright 3pc\box\@ne
  \else % if the float IS NOT a figure...
    \moveright 3pc\box\@ne
    \nobreak
    \vskip\belowcaptionskip
  \fi
\relax
}
\makeatother
\def\writefig#1 #2 #3 {\rlap{\kern #1 truecm
\raise #2 truecm \hbox{#3}}}

\title{Convergence in law for Complex Gaussian Multiplicative Chaos in phase III}

\author{Hubert Lacoin}
\address{
  IMPA, Institudo de Matem\'atica Pura e Aplicada, Estrada Dona Castorina 110
Rio de Janeiro, CEP-22460-320, Brasil. 
}

\begin{document}

 \begin{abstract}
 Gaussian Multiplicative Chaos (GMC) is informally defined as  a random measure $e^{\gamma X} \dd x$ where $X$ is Gaussian field on $\bbR^d$ (or an open subset of it) whose correlation function is of the form 
 $ K(x,y)= \log \frac{1}{|y-x|}+ L(x,y),$
 where $L$ is a continuous function $x$ and $y$ and $\gamma=\alpha+i\beta$ is a complex parameter. In the present paper, we consider the case $\gamma\in \cP'_{\mathrm{III}}$
 where 
 $$  \cP'_{\mathrm{III}}:= \{ \alpha+i \beta \ : \alpha,\gamma \in \bbR , \ |\alpha|<\sqrt{d/2}, \ \alpha^2+\beta^2\ge d  \}.$$
 We prove that if $X$ is replaced by the approximation $X_\gep$ obtained by convolution with a smooth kernel, then  $e^{\gamma X_\gep} \dd x$, when properly rescaled, has an explicit nontrivial limit in distribution when $\gep$ goes to zero. This limit does not depend on the specific convolution kernel which is used to define $X_{\gep}$ and can be described as 
 a complex Gaussian white noise with a random intensity given by  a real GMC associated with parameter $2\alpha$.
   \\[10pt]
 \end{abstract}

\maketitle

%  \tableofcontents
 
 \section{Model and results}

 \subsection{The exponential of a log-correlated field via convolution approximation}
 
Given an open domain $\cD\subset \bbR^d$, we consider $K: \cD^2 \to \bbR$ to be a positive definite kernel on $\cD$ which admits a decomposition in the following form
 \begin{equation}\label{fourme}
  K(x,y):= L(x,y)+\log \frac{1}{|x-y|},
 \end{equation}
where $L$ is a continuous function.
A kernel $K$ is  positive definite if for any bounded continuous function $\rho$ with compact support on $\cD$
\begin{equation}\label{ladefpos}
 \int_{\cD^2 }  K(x,y) \rho(x)\rho(y)\dd x \dd y\ge 0.
\end{equation}
 We want to consider a Gaussian field $X$ with covariance $K$ and find a way to make sense of the distribution $e^{\gamma X(x)} \dd x$ when $\gamma=\alpha+i\beta$ is a complex parameter.
Such an exponentiation a $\log$-correlated field is what is called Gaussian Multiplicative Chaos (GMC) and has been the focus of a large amount of mathematical work first in the case $\gamma\in \bbR$ (see \cite{RVreview} for a review and references) and more recently  $\gamma \in \bbC$ (see \cite{junnila2019regularity,  junnila2019, junnila2018,  lacoin2020, LRV19, LRV15} and references therein).
The standard way to define GMC is to consider $X$ taking value in the field of distribution and then use an approximation of $X$ and a passage to the limit to define its exponential. 

\medskip

Since $K$ is infinite on the diagonal, it is not possible to define directly a Gaussian field indexed by $\cD$ with covariance function $K$.
We consider thus a distributional Gaussian field, that is, a field indexed by a set of signed measure. Like in \cite{Natele}, we define $\cM^+_K$ to be the set of positive Borel measures on $\cD$ such that 
\begin{equation}
 \int_{\cD^2}  |K(x,y)|\mu(\dd x) \mu(\dd y) <\infty
\end{equation}
and let $\cM_K$ be the space of signed measure spanned by $\cM^+_K$
\begin{equation}
 \cM_K:=\left\{ \mu_+- \mu_- \ : \ \mu_+, \mu_- \in \cM^+_K \right\}.
\end{equation}
We define $\bar K$ as the following quadratic form on  $\cM_K$  
\begin{equation}\label{hatK}
 \bar K(\mu,\mu')=
\int_{\cD^2}  K(x,y)\mu(\dd x) \mu'(\dd y),
\end{equation}
and finally define $X= (\langle X, \mu \rangle)_{\mu\in \cM_K}$ as a the random field indexed by $\cM_{K}$ with covariance kernel given by $\bar K$.

\medskip

The distributional field $X$ can be approximated by a sequence of  functional fields - that is,  fields indexed by a subset of $\cD$ - by the mean of convolution with smooth kernels (we have to consider a strict subset of $\cD$ to avoid boundary effects).
 Consider $\theta$ a nonnegative $C^{\infty}$ function whose compact support is included in $B(0,1)$ the $d$-dimensional  Euclidean ball of radius one,  $\int_{B(0,1)} \theta(x) \dd x=1.$
We define for $\gep>0$,
$\theta_{\gep}:=\gep^{-d} \theta( \gep^{-1}\cdot)$
and consider the convoluted version of $X$ on the set
\begin{equation}\label{cdgep}
\cD_{\gep}:=\{ x\in \cD \ : \ \min_{y\in  \bbR^d \setminus \cD} |x-y| \ge  2\gep \} 
\end{equation}
(or $\cD_{\gep}=\bbR^d$ convention if $\cD=\bbR^d$). It is defined by 
\begin{equation}
 X_{\gep}(x):= \langle X, \theta_{\gep}(x-\cdot) \rangle
\end{equation}
where  the function $\theta_{\gep}(x-\cdot)$ is identified with the measure $\theta_{\gep}(y-x)\dd y$ on $\cD$.
With this definition one can check that $X_{\gep}(x)$ has covariance 
\begin{equation}\label{labig}
K_{\gep}(x,y):=\bbE[X_{\gep}(x)X_{\gep}(y)]=
\int_{\bbR^{2d}} \theta_{\gep}(x-z_1) \theta_{\gep}(y-z_2)
K(z_1,z_2)\dd z_1 \dd z_2.
\end{equation}
We simply write $K_{\gep}(x)$ when $x=y$.
 Note that since $K_{\gep}$ is infinitely differentiable, by Kolmogorov's criterion (see e.g.\  \cite[Theorem 2.9]{legallSto}),  there exists a version of $X_{\gep}$ which is continuous
 in $x$. Considering this version of the field, we can make sense of integrals of measurable functionals of $X_\gep(\cdot)$.
 This allows to define, for any $f\in C^\infty_c(\cD)$ (the set of \textit{real valued} infinitely differentiable functions with  compact  support), the quantity
 \begin{equation}
    M^{(\gamma)}_{\gep}(f):=\int_{\bbR^d} f(x) e^{\gamma X_{\gep}(x)-\frac{\gamma^2}{2}K_\gep(x)} \ind_{\cD_{\gep}}(x)\dd x.
 \end{equation}
 The restriction  to $\cD_{\gep}$ ensures that $K_\gep$ is well defined and uniformly bounded, which is convenient. It disappears when $\gep$ goes to zero, since for $\gep$ sufficiently small the support of $f$ is included in $\cD_{\gep}$.
The question of focus in the present paper is the existence 
of a nontrivial limit of the distribution $M^{(\gamma)}_{\gep}(\cdot )$ when $\gep$ tend to zero (possibly with a rescaling by a factor depending on $\gep$). Such a limit gives natural 
 interpretation for the formal distribution $e^{\gamma X} \dd x$.

 \subsection{The case of real GMC}

The  case when the parameter in the exponentiation is real (in that case we write it as $\alpha$ instead of $\gamma$) has been extensively studied, starting with the work of Kahane \cite{zbMATH03960673} (see for instance \cite{Natele,MR2819163,robertvargas}, we refer to the introduction in \cite{Natele} for a detailed chronological account). 
 These works established that when $\alpha \in (-\sqrt{2d},\sqrt{2d})$ then $M^{(\alpha)}_{\gep}$ converges to a non trivial limit. As this result, together its  variant Theorem \ref{zabbid2} presented in the next section, play a pivotal role in our proof so we state it in full details.

\begin{theorema}\cite[Theorem 1.1]{Natele}\label{zabbid}
 For any $\alpha \in\bbR$ with $|\alpha |<\sqrt{2d}$, then there exists a random distribution $M^{(\alpha)}_{0}$ such that for every $\theta$  and every  $f\in C^{\infty}_c(\cD)$, we have the following convergence in $\bbL_1$
  $$   \lim_{\gep \to 0}M^{(\alpha)}_{\gep}(f)= M^{(\alpha)}_{0}(f). $$
 The distribution $M^{(\alpha)}_{0}$ is a locally finite measure whose support is dense in $\cD$. The limit does not depend on the convolution kernel $\theta$ used in the definition of $M^{(\alpha)}_{\gep}$.
\end{theorema}

The condition $|\alpha |<\sqrt{2d}$ is optimal : when $|\alpha|\ge \sqrt{2d}$, then $\lim_{\gep\to 0} M^{(\gamma)}_{\gep}(f)=0$ for all $f$. 
When $\alpha=\pm \sqrt{2d}$, one can still obtain a nontrivial limit by adding a scaling factor: the measure  $\sqrt{\log (1/\gep)} M^{(\pm \sqrt{2d})}_{\gep}(\cdot)$ converges in probability to a positive measure referred to as the \textit{critical} multiplicative chaos (see \cite{MR3262492} for a first derivation \cite{junnila2017ejp} for uniqueness of the limit and \cite{powell2020critical} for an up-to-date review). When $|\alpha|>\sqrt{2d}$, one should still obtain a convergence after an adequate rescaling but the convergence is of a different nature since it only holds in distribution (see \cite{madaule2016} for such a result with $X_{\gep}$ replaced by a martingale sequence of approximation similar to the one considered in Section \ref{martingaledecompo} of the present paper).

\subsection{Complex white noise with random intensity given by a Real GMC}

In order to introduce our limit, let us introduce the notion of the complex white noise with a
random intensity.

\medskip

For $\gamma\in \cP'_{\mathrm {III}}$ we define $\mathfrak M^{(\gamma)}$ to be a complex white noise with intensity measure given by 
 $M^{(2\alpha)}_0(e^{|\gamma|^2L}\cdot)$ where $\alpha\in (-\sqrt{2d},\sqrt{2d})$.
It is a random linear form which is constructed jointly with $X$, on an extended space (we let $\bbP$ denote the corresponding probability).
Conditionally to $X$, for $f\in C^{\infty}_c(\cD)$,  $\mathfrak M^{(\gamma)}(f)$ is a complex Gaussian random variable, with independent real and imaginary parts whose variances equal  
 $$ M^{(2\alpha)}_0(e^{|\gamma|^2L} f^2)= \int_{\cD} e^{|\gamma|^2 L(x,x)}f(x)^2 M^{(2\alpha)}_0(\dd x).$$
Formally, $\mathfrak M^{(\gamma)}(\cdot)$ is a random process indexed by $C^{\infty}_c(\cD)$
whose joint law with $X$, which we denote by  $\bar \bbP$, satisfies
 for any $n,m\ge 1$,  $\mu_1,\dots,\mu_m \in \cM_K$, $f_1,\dots,f_n \in C^{\infty}_c(\bbR^d)$
 and any bounded measurable function $F$ on $\bbR^{m}\times \bbC^{n}$
\begin{multline}\label{cgwn}
 \bar\bbE \left[ F\left( (\langle X, \mu_i \rangle)_{i=1}^m, ( \mathfrak M^{(\gamma)}(f_j))^n_{j=1}  \right)  \right]\\= \bbE\otimes \bE \left[  F\left( (\langle X, \mu_i \rangle)_{i=1}^m, \Sigma[\gamma,X, (f_j)_{j=1}^n]\cdot \cN_n  \right) \right] 
\end{multline}
where $\cN_n$ is an $n$ dimensional vector (with probability distribution $\bP$) whose coordinate are IID standard complex Gaussian variables, and 
$\Sigma[\gamma,X, (f_j)_{j=1}^n]$ is the positive definite square root of the matrix
$$
 \left(M^{(2\alpha)}(e^{|\gamma|^2 L}f_if_j)\right)_{i,j=1}^n.$$

 \medskip
 
 The process $\mathfrak M^{(\gamma)}$ can be seen as a random distribution. More precisely, there exists a version of the process $\mathfrak M^{(\gamma)}$ taking values in the local Sobolev space $H^{-u}_{\mathrm{loc}}(\cD)$ with $u>d/2$ (see definition \eqref{locsob} below).
This regularity for $\mathfrak M^{(\gamma)}$ can for instance be obtained by  the combining  Proposition \ref{tiggh} and Theorem \ref{main} proved below.

\subsection{Our main result}

Our main theorem concerns the convergence of $M^{(\gamma)}_{\gep}$ for complex values $\gamma$. More precisely we consider $\gamma$ in the following range of parameters
 \begin{equation}\label{pIIIprime}
 \cP'_{\mathrm{III}}:= \{ \alpha+i \beta \ : \alpha,\beta \in \bbR, \ |\alpha|<\sqrt{d/2},\  \alpha^2+\beta^2\ge d  \}.
 \end{equation} 
We require an assumption on the regularity of the function $L$ present in \eqref{fourme} (a condition which is also present for papers investigating the subcritical complex chaos \cite{junnila2019,lacoin2020} for a similar reasons).
Let us recall the definition for the Sobolev space with index $s\in \bbR$ on $\bbR^k$ 
which is the Hilbert space of complex valued function associated with the norm 
\begin{equation}
 \| \varphi \|_{H^{s}(\bbR^k)}:= \left( \int_{\bbR^{k}} (1+|\xi|^2)^s |\hat \varphi(\xi)|^2 \dd \xi \right)^{1/2},
\end{equation}
where $\hat \varphi(\xi)$ is the Fourier transform of $\varphi$  defined for $\varphi\in C^{\infty}_c(\bbR^k)$ by
$$ \hat \varphi(\xi)= \int_{\bbR^{k}} e^{i\xi x} \varphi(x)\dd x.$$
Now for $U\subset \bbR^k$ open, the local Sobolev space $H^{s}_{\mathrm{loc}}(U)$ denotes the function which belongs to  $H^{s}(U)$ after multiplication by an arbitrary smooth function with compact support
\begin{equation}\label{locsob}
H^{s}_{\mathrm{loc}}(U):= \left\{  \varphi : U \to \bbR  \ | \ \rho\varphi\in
 H^{s}(\bbR^d) \text{ for all } \rho\in C^{\infty}_c(U)\right\},
\end{equation}
where above $\rho\varphi$ is identified with its extension by zero on $\bbR^k$.
We are going to assume that the covariance kernel $K$ is of the form \eqref{fourme} with $L\in H^{s}_{\mathrm{loc}}(\cD^2)$ for some exponent $s>d$.

\medskip
  Before stating the result we need to introduce some notation.
Let us define the function $\ell_{\theta}$  on $\bbR^d$, obtained by convoluting $z\mapsto \log 1/|z|$ twice with $\theta$, that is
\begin{equation}\label{ltheta}
 \ell_{\theta}(z):= \int_{\bbR^d} \log\left(\frac{1}{|z+z_1-z_2|}\right) \theta(z_1) \theta(z_2) \dd z_1 \dd z_2.
\end{equation}
and set

\begin{equation}    v(\gep,\theta,\gamma):=
 \begin{cases}
\gep^{\frac{|\gamma|^2-d}{2}} \left(\frac{1}{2}\int_{\bbR^d} e^{|\gamma|^2\ell_{\theta}(z)} \dd z\right)^{-1/2} \quad &\text{ if } |\gamma|>\sqrt{d},\\
    \left(\frac{\pi^{\frac{d}{2}}}{\gG(d/2)} \log \frac 1 \gep\right)^{-1/2} \quad  &\text{ if } |\gamma|=\sqrt{d},
 \end{cases}
 \end{equation}
(the quantity 
$\frac{\pi^{\frac{d}{2}}}{\gG(d/2)}$  corresponds to half of the volume of the $(d-1)$-dimensional sphere).
Our result establishes that for $\gamma\in \cP'_{\mathrm{III}}$
when $\gep$ goes to zero $v(\gep,\theta,\gamma) M^{(\gamma)}_{\gep}$ converges in distribution towards
a complex Gaussian white noise, with random  intensity is
$M^{(2\alpha)}_0\left(e^{|\gamma|^2L}  \cdot\right)$.

\begin{theorem}\label{mainall}
Let $X$ be a Gaussian random field on $\cD$ whose covariance kernel is of the form \eqref{fourme}
with $L\in H^{s}_{\mathrm{loc}}(\cD^2)$ for some $s>d$.
For $\gamma \in \cP'_{\mathrm{III}}$,  $u>d/2$,
 the distribution $v(\gep,\theta,\gamma) M^{(\gamma)}_{\gep}$  converges in law, for the $H^{-u}_{\mathrm{loc}}(\cD)$ topology, towards $\mathfrak M^{(\gamma)}$ defined in \eqref{cgwn}. 
More precisely we have the following joint convergence in law
\begin{equation}\label{stableconv}
 (X, v(\gep,\theta,\gamma) M^{(\gamma)}_{\gep})   \quad \stackrel{\gep \to 0}{ \Longrightarrow}
\quad  (X , \mathfrak M^{(\gamma)}).
 \end{equation}

\end{theorem}

The convergence \eqref{stableconv} indicates that the limit of $\mathfrak M^{(\gamma)}$ informally splits into two parts: the intensity of the noise which is determined by the realization of the field $X$ and an additional Gaussian randomness which is independent from $X$. In particular \eqref{stableconv} implies that $v(\gep,\theta,\gamma) M^{(\gamma)}_{\gep}$ \textit{does not} converge to a limit in probability.
It is a particular case of \textit{stable convergence} (see \cite[VIII-Section 5c]{jacodsh}).

To prove Theorem \ref{mainall} we prove separetely the tightness of $v(\gep,\theta,\gamma) M^{(\gamma)}_{\gep}$  in   $H^{-u}_{\mathrm{loc}}(\cD)$ if $u>d/2$ and the convergence of the finite dimensional marginals.  
The proof of the tightness result below, while a bit technical, follows a standard approach and for this reason is given in  Appendix \ref{tightness}.

\begin{proposition}\label{tiggh}
Under the assumptions of Theorem \ref{mainall}, given  $\rho\in C^{\infty}_c(\cD)$. 
The random sequence $(v(\gep,\theta,\gamma) M^{(\gamma)}_{\gep}(\rho  \ \cdot))_{\gep\in(0,1)}$  is tight in 
 $H^{-u}(\bbR^d)$ for any $u>d/2$. 
\end{proposition}
 To prove the convergence of finite dimensional marginals of  
$(X, v(\gep,\theta,\gamma) M^{(\gamma)}_{\gep})$, using Lévy's Theorem, it is sufficient to prove the pointwise convergence of the Fourier  transform of real the real valued random vectors of the type
 $$\left( (\langle\mu_j,X\rangle)_{j=1}^m, \big(v(\gep,\theta,\gamma)\mathfrak{Re}(M^{(\gamma)}_{\gep}(f_k)\big)_{k=1}^n, \big(v(\gep,\theta,\gamma)\mathfrak{Im}(M^{(\gamma)}_{\gep}(f_k)\big)_{k=1}^n
 \right)
 $$
 for $\mu_1,\dots,\mu_m \in \cM_K$, $f_1,\dots,f_n\in C^{\infty}_c(\cD)$,
  where here and in the remainder of the paper $\mathfrak{Re}$ and $\mathfrak{Im}$ are used to denote the real part of a complex number.
 This amounts to checking the convergence of 
$
 \bbE \left[e^{i \langle X,\mu\rangle + i v(\gep,\theta,\gamma)M^{(\gamma)}_{\gep}(f,\go)} \right]$
 for every $\mu\in \cM_K$, $f\in C^{\infty}_c(\cD)$ and $\go\in [0,2\pi)$, where
 \begin{equation}\begin{split}
 M^{(\gamma)}_{\gep}(f,\go)&:= \mathfrak{Re}\left( e^{-i\go}  M^{(\gamma)}_{\gep}(f)\right)\\&=
 \int_{\cD_{\gep}} f(x)e^{\alpha X_{\gep}(x)+\frac{\beta^2-\alpha^2}{2}K_{\gep}(x)} \cos(\beta(X_{\gep}(x)-2\alpha\beta K_{\gep}(x)-\go))
 \dd x.
\end{split}
 \end{equation} 
We let $\cF_X$ denote the $\sigma$-algebra generated by the process $X$
\begin{equation}\label{cfx}
 \cF_X:= \sigma( \langle X , \mu \rangle , \mu\in \cM_K ).
\end{equation}
Note that from the definition  \eqref{cgwn} we have for every $f\in C^{\infty}_c(\cD)$ and $\go\in[0,2\pi)$
\begin{equation}\label{ftrans}
\bar \bbE \left[ e^{i \mathfrak{Re}(e^{-i\go}\mathfrak M^{(\gamma)}(f)}) \ | \ \cF_X \right]= e^{- \frac{1}{2}M^{(2\alpha)}_{0}(e^{|\gamma|^2 L}f^2)} .
\end{equation}
The following result is the main technical achievement of the paper.

\begin{theorem}\label{main}
Under the assumption of Theorem \ref{mainall}, given $f\in C^{\infty}_c(\cD)$, $\go\in[0,2\pi)$, and $\mu \in \cM_K$  we have
\begin{equation}\begin{split}
 \lim_{\gep \to 0} \bbE \left[e^{i \langle X,\mu\rangle + i v(\gep,\theta,\gamma)M^{(\gamma)}_{\gep}(f,\go)} \right]&=
 \bar \bbE\left[ e^{i \langle X,\mu\rangle +  i \mathfrak{Re}(e^{-i\go} \mathfrak M^{(\gamma)}(f)} )\right]\\
 &=
 \bbE \left[ e^{i \langle X,\mu\rangle -\frac{1}{2}M^{(2\alpha)}_{0}(e^{|\gamma|^2L}f^2)} \right].
\end{split}
 \end{equation}
 
\end{theorem}

\begin{rem}
Note that as in Theorem \ref{zabbid} or its extension to the complex case \cite[Theorem 2.1]{lacoin2020} the obtained limit does not depend on the convolution kernel used to define $X_{\gep}$ (the scaling factor $v(\gep,\theta,\gamma)$ does, although not when $|\gamma|^=\sqrt{d}$).
\end{rem}

\subsection{A review of results on complex GMC}

Let us now try to give some perspective on our results by relating it to the existing literature on complex GMC. The set $\cP'_{\mathrm{III}}$ in \eqref{pIIIprime} corresponds in fact - up to boundary - to one of the phases in a diagram which contains three. Let us introduce these three phases for the needs of the discussion
\begin{equation}\begin{split}
            \label{phasediagg}
 \cP_{\mathrm{sub}}&:=\left\{\alpha+i \beta \ : \alpha^2+\beta^2<d\right\} \cup\left\{  \alpha+i \beta  \ : \ \alpha \in (\sqrt{d/2},\sqrt{2d}) \ ; \  |\alpha|+|\beta|<\sqrt{2d} \ \right\},\\
  \cP_{\mathrm{II}}&:= \left\{\alpha+i \beta \ : |\alpha|+|\beta|>\sqrt{2d} \ ; \  |\alpha|> \sqrt{d/2} \right\},\\
 \cP_{\mathrm{III}}&:= \left\{\alpha+i \beta \ : \alpha^2+\beta^2> d \ ; \ |\alpha| < \sqrt{d/2} \right\}.
                \end{split}
\end{equation}
This phase diagram appears in 
in \cite{LRV15} in the context of the study of complex GMC is defined by considering
 \begin{equation}\label{defmab}
 M^{(\alpha,\beta)}_{\gep}(\dd x):= e^{\alpha X_{\gep}+ i\beta Y_{\gep}- \frac{\alpha^2-\beta^2}{2}K_{\gep}(x)} \dd x
 \end{equation}
 where $X$ and $Y$ are independent $\log$-correlated fields of covariance $K$ \footnote{To be completely accurate, the field $X_{\gep}$ considered  in \cite{LRV15}  is not a convolution but rather a martingale approximation of the type considered in  Section \ref{martingaledecompo}. This difference is not   relevant for  the present  discussion.}, and also for related models such as 
complex Multiplicative Cascades \cite{DES93} complex Random Energy Model \cite{KK14} or complex branching Brownian Motion \cite{HK15, HK18}.
Each of the phases in \eqref{phasediagg} is conjectured to correspond to a different scaling regime for $M^{(\gamma)}_{\gep}$

\subsubsection*{The subcritical phase $\cP_{\mathrm{sub}}$}
When $\gamma\in \cP_{\mathrm{sub}}$ is has been proved that  $M^{(\gamma)}_{\gep}$ converges 
to a random distribution.
More precisely is has been proved in \cite{junnila2019} (with the assumption $L\in H^s_{\mathrm{loc}}(\cD^2)$, for $s>d$) that the random distribution  $M^{(\alpha)}_{0}$ of Theorem \ref{zabbid} has a unique analytic continuation on the domain  $\cP_{\mathrm{sub}}$.
In \cite{lacoin2020}, it was proved (under the same assumption) that  this analytic continuation is the limit of  $M^{(\gamma)}_\gep$ for any choice of convolution kernel $\theta$ , extending Theorem \ref{zabbid} to the full region  $\cP_{\mathrm{sub}}$.
Convergence of  $M^{(\alpha,\beta)}_{\gep}$, recall \eqref{defmab}, for  $(\alpha,\beta)\in \cP_{\mathrm{sub}}$ (identified with a subset of $\bbR^2$) is also established in \cite{lacoin2020} (and earlier in \cite{LRV15} for the martingale approximation).

\subsubsection*{The glassy phase $\cP_{\mathrm{II}}$}

When $\gamma \in \cP_{\mathrm {II}}$, it is conjectured that  in the limit when $\gep\to 0$ the distribution of $M^{(\gamma)}_\gep$ is supported by small neighborhood of the  points where  $X_{\gep}$ is close to be maximized, yielding an atomic distribution (a countable sum of weighted Dirac masses) in the limit.
In this case the right  scaling should be $(\log 1/\gep)^{\frac{3\alpha}{2}}\gep^{\sqrt{2d}\alpha-d} M^{(\gamma)}_\gep(\dd x)$.
This phenomenon is called freezing and has been proved in \cite{HK15, madaule2016} for the complex exponential of Branching Brownian Motion, but it remains a challenging conjecture for complex GMC
(both for $M^{(\alpha,\beta)}_{\gep}$ and $M^{(\gamma)}_{\gep}$).

\subsubsection*{The third phase $\cP_{\mathrm{III}}$}
When $\gamma\in \cP_{\mathrm{III}}$ the fluctuations of $X_{\gep}$  makes 
the phases of $e^{i\beta X_{\gep}(x)}$ decorelate even on small scale, and this accounts for the appearance of a white noise appear in the limit.  The intensity of the corresponding white noise has to be given by the limit the square of the modulus of the local variations, that is $e^{2\gamma X_{\gep}} \dd x$.
The convergence of $\gep^{|\gamma|^2-d} M^{(\alpha,\beta)}_\gep$ (recall \eqref{defmab}) was established in \cite{LRV15}. The proof relied on the computation of all the conditional moments  of $M^{(\alpha,\beta)}_\gep$ when conditioning w.r.t.\ to the field $X$.  This approach is heavily relying on the independence of $X$ and $Y$ and cannot be adapted to the present context.

\medskip

In this paper, partly inspired the techniques used in \cite{LRV19} to study the Sine-Gordon model (which corresponds to the case $\alpha=0$) we take a completely different approach which relies on convergence of martingale brackets after using a martingale decomposition.

\subsection{Open questions}

Note that our result Theorem \ref{mainall} does not only compute the scaling limit of $M^{(\gamma)}_{\gep}$ in $\cP_{\mathrm{III}}$ but also on the frontier between $\cP_{\mathrm{III}}$ and $\cP_{\mathrm{sub}}$ (recall the definition of $\cP'_{\mathrm{III}}$   \eqref{pIIIprime}).
Let us discuss here shortly what should occur on the rest of the frontier between $\cP_{\mathrm{III}}$ and other phases.

\medskip

To formulate this conjecture, let us introduce the critical real multiplicative chaos, which corresponds to the point $|\alpha|= \sqrt{2d}$ and $\beta=0$.
It has been proved \cite{MR3262492,junnila2017ejp} that while  $M^{(\pm \sqrt{2d})}_\gep(\dd x)$
converges to zero, we obtain a non trivial limit in probability after rescaling by an appropriate factor. More precisely we have
$$\lim_{\gep\to 0}(\log 1/\gep)^{1/2} M^{(\pm \sqrt{2d})}_\gep(\dd x)=:M^{(\pm \sqrt{2d})} (\dd x). $$
The measure $M^{(\pm \sqrt{2d})}$ is referred to
 to as the \textit{critical} multiplicative chaos.

\subsubsection*{The frontier $\cP_{\mathrm {II}}/\cP_{\mathrm {III}}$}
When $|\alpha|=\sqrt{d/2}$, $|\beta|>\sqrt{d/2}$, it is natural 
to conjecture that $M^{(\gamma)}_\gep$ properly renormalized should converges to a white noise whose intensity is given by $M^{(2\alpha)}$.
Such a result has been proved in \cite{LRV15} for the chaos given in \eqref{defmab}.
The proper renormalization to consider in that case should be $(\log 1/\gep)^{\frac{1}{4}}\gep^{\frac{|\gamma|^2-d}{2}} M^{(\gamma)}_\gep$.

\subsubsection*{The triple point $|\alpha|=|\beta|=\sqrt{d/2}$}
This should be similar to the $\cP_{\mathrm {II}}/\cP_{\mathrm {III}}$ frontier though possibly more technical to handle.
In that case $(\log 1/\gep)^{-\frac{1}{4}}M^{(\gamma)}_\gep$ should converge to a complex white noise with intensity given by $M^{(2\alpha)}$.

\subsection*{First hints on the organization of the paper}

 Proposition \ref{tiggh} is proved in Appendix \ref{tightness}.
 The proof of Theorem \ref{main} which is the main technical achievement of the paper rely on a martingale decomposition of $M^{(\gamma)}_{\gep}$ which   is introduced in Section \ref{martingaledecompo}. Additional details on this decomposition are needed to describe the remainder of the organization of the paper,  a more detailed picture is given in Section\ref{lorga}.

\subsection*{Notation}
Let us list here a few convention adopted in the paper.
If $G$ is a generic function of two variables we will write $G(x)$ for $G(x,x)$.
If $(J_s)_{s\ge 0}$ is a continuous function or a random process indexed by $s$ we set 
\begin{equation}\label{intervalnotation}
 J_{[a,b]}:=J_b-J_a
\end{equation}
The letter $C$ is used to denote generic positive constants used in the computation. The value of $C$ is allowed to change from one equation to another  within the same proof.
 
\section{Decompositing the proof of Theorem \ref{main}}\label{martingaledecompo}

 \subsection{Star-scale invariant kernels}

We are going to say that the kernel $K$ has \textit{a star-scale invariant part}  (with kernel $\kappa$) if it can be written in the form
\begin{equation}\label{ilaunestar}
 K(x,y)=K_0(x,y)+ \int^{\infty}_{0} \kappa(e^{t}|x-y|)\dd t
\end{equation}
where $K_0(x,y)$ is a bounded H\"older continuous positive definite kernel on $\cD$ (recall \ref{ladefpos}) and the function $\kappa: \bbR_+ \to \bbR$, satisfies the following assumptions:
\begin{itemize}
 \item [(i)] $\kappa$ is Lipshitz-continuous and non-negative,
 \item [(ii)] $\kappa(0)=1$, $\kappa(r)=0$ for $r\ge 1$.
 \item [(iii)]$(x,y)\mapsto \kappa(|x-y|)$ defines a positive definite function on $\bbR^d\times \bbR^d$. 
\end{itemize}
\noindent Note that if $K$ satisfies  \eqref{ilaunestar} then
\begin{equation}
 L(x,y):= K(x,y)+\log |x-y|,
\end{equation}
can be extended to a continuous function on $\cD^{2}$, so that $K$  having a star-scale invariant part
implies that it can be written in the form \eqref{fourme}.

\medskip

The converse is not true and there are  positive definite kernel $K$ of the type given in Equation \eqref{fourme} that cannot be decomposed as in \eqref{ilaunestar} for any choice of $\kappa$.
However if $L\in H^s_{\mathrm{loc}}(\cD^2)$ for some $s>d$, then  $K$ can be approximated very well by a kernel of the type \eqref{ilaunestar}.
This is the content of the following result  (proved in Appendix \ref{approxapp}).

\begin{lemma}\label{lesasump}
 Given $K$ a covariance kernel on $\cD$ of the form \eqref{fourme} with $L\in H^s_{\mathrm{loc}}(\cD^2)$ for $s>d$
 , $\cD'$ a bounded open set whose closure satisfies   $\bar \cD'\subset \cD$ and $\delta>0$
  there exists a kernel $K^{(\delta)}$  of the form \eqref{ilaunestar} on $\cD'$ such that 
 \begin{itemize}
  \item [(A)]  For all $x,y \in \cD', \quad   |K^{\delta}(x,y)-K(x,y)|\le \delta$.
\item [(B)] $\Delta^{(\delta)}(x,y)=K^{(\delta)}(x,y)-K(x,y)$ is a positive definite kernel on $\cD'$.
 \end{itemize}

\end{lemma}

In order to prove Theorem \ref{main}, the most important step is to prove the convergence of $M^{(\gamma)}_{\gep}(f,\go)$ for a field whose covariance kernel satisfies  the assumption in \eqref{ilaunestar} and use our approximation Lemma \ref{lesasump} to conclude.

\begin{proposition}\label{withypo}
Given a Gaussian field $X$ defined on $\cD$ whose covariance kernel satisfies \eqref{ilaunestar},
a  convolution kernel $\theta$, and $\gamma \in \cP'_{\mathrm{III}}$, we have for every $f\in C^{\infty}_c(\cD)$, $\go\in[0,2\pi)$, and $\mu \in \cM_K$  we have
\begin{equation}
 \lim_{\gep \to 0} \bbE \left[e^{i \langle X,\mu\rangle + i v(\gep,\theta,\gamma)M^{(\gamma)}_{\gep}(f,\go)} \right]=
 \bbE \left[ e^{-i  \langle X,\mu\rangle-\frac{1}{2}M^{(2\alpha)}_{0}(e^{|\gamma|^2L}f^2)} \right].
 \end{equation}
 
\end{proposition}
\noindent The proof of Proposition \ref{withypo} is technical and require several steps. We provide a detailed road map at the end of this section. Let us first explain how we deduce our main result from it.

\subsection{Proving Theorem \ref{main} from Proposition \ref{withypo}}
 
Let us fix $f\in C^{\infty}_c(\cD)$ and $\go\in (0,2\pi]$ and $\mu\in \cM_K$. Given $\eta>0$
we consider a bounded open set $\cD'$ which includes the support of $f$ and which is such that
\begin{equation}\label{ddapprox}
\int_{(\cD \setminus \cD')^2}  K(x,y)\mu(\dd x)\mu(\dd y)\le \eta
\end{equation}
and whose  topological closure satisfies $\bar \cD'\subset \cD$.
 We let $K^{(\delta)}$ be a kernel satisfying the assumptions (A) and (B) of Lemma \ref{lesasump}. Given $\delta>0$,
 we can construct two  Gaussian fields $X$ and $X^{(\delta)}$ indexed by $\cD$ and $\cD'$ on the same probability space in such a way that the field 
 $Z^{(\delta)}:= X^{(\delta)}-X$ defined on $\cD'$ is independent of $X$ and has covariance  $\gD^{(\delta)}(x,y)$.
 Letting ${M}^{(\gamma,\delta)}_{\gep}$ denote  the exponential of the convoluted field  $X^{(\delta)}_{\gep}$, and $\mu'$ denote the restriction of $\mu$ to $\cD'$ we have from Proposition
 \ref{withypo}
 
 \begin{equation}\label{proxilimit}
   \lim_{\gep \to 0} \bbE \left[ e^{i\langle X^{(\delta)}, \mu'\rangle+i v(\gep,\theta,\gamma){M}^{(\gamma,\delta)}_{\gep}(f,\go)} \right]=
   \bbE \left[e^{i\langle X^{(\delta)}, \mu'\rangle- \frac 1 2 {M}^{(2\alpha,\delta)}_{0}(e^{|\gamma|^2 L}f^2)} \right].
 \end{equation}
 Now in order to obtain a conclusion for the exponential of the original field $X$
 we have to replace $X^{(\delta)}$ by $X$ and $\mu'$ by $\mu$ in both sides of the above convergence.
 Using Jensen's inequality and the triangle inequality we have
 \begin{multline}\label{ladif1}
  \left|\bbE \left[ e^{i\langle X, \mu\rangle+i v(\gep,\theta,\gamma){M}^{(\gamma)}_{\gep}(f,\go)} -  e^{i\langle X^{(\delta)}, \mu'\rangle+i v(\gep,\theta,\gamma){M}^{(\gamma,\delta)}_{\gep}(f,\go)} \right]\right|
  \\ \le
+
\bbE \left[ \left|  e^{i v(\gep,\theta,\gamma){M}^{(\gamma,\delta)}_{\gep}(f,\go)} 
 -   e^{i v(\gep,\theta,\gamma){M}^{(\gamma)}_{\gep}(f,\go)} \right| \right]
 +  
   \bbE \left[ |e^{i\langle X^{(\delta)}, \mu'\rangle}-e^{i\langle X, \mu\rangle}|  \right].
 \end{multline}
 and in the same manner 
 \begin{multline}\label{ladif2}
  \left|\bbE \left[ e^{i\langle X, \mu\rangle- \frac 1 2 {M}^{(2\alpha)}_{0}(e^{|\gamma|^2 L}f^2)}  -  e^{i\langle X^{(\delta)}, \mu'\rangle- \frac{1}{2} {M}^{(2\alpha,\delta)}_{0}(e^{|\gamma|^2 L}f^2)} \right]\right|
  \\ \le
+
\bbE \left[ \left|   e^{- \frac 1 2 {M}^{(2\alpha)}_{0}(e^{|\gamma|^2 L}f^2)}- e^{- \frac 1 2 {M}^{(2\alpha,\delta)}_{0}(e^{|\gamma|^2 L}f^2)}  \right| \right]
 +  
   \bbE \left[ |e^{i\langle X^{(\delta)}, \mu'\rangle}-e^{i\langle X, \mu\rangle}|  \right].
 \end{multline}
In both r.h.s.\ of  \eqref{ladif1}-\eqref{ladif2} the second term is easier to control. As $u\mapsto e^{iu}$ is Lipshitz we have
  \begin{equation}\label{proxilimat2}
  \bbE \left[ |e^{i\langle X^{(\delta)}, \mu'\rangle}-e^{i\langle X, \mu\rangle}|  \right]\le
  \bbE\left[ |\langle X^{(\delta)}, \mu'\rangle-\langle X, \mu\rangle| \right]
  \le   \bbE\left[ (\langle X^{(\delta)}, \mu'\rangle-\langle X, \mu\rangle)^2 \right]^{1/2}.
 \end{equation}
The variance of can be computed explicitly, we have (recall \eqref{ddapprox})
\begin{multline}
  \bbE\left[ (\langle X^{(\delta)}, \mu'\rangle-\langle X, \mu\rangle)^2 \right]
  = \int_{(\cD')^2} \!\!\!\!\! \Delta^{(\delta)}(x,y) \mu(\dd x)\mu(\dd y)
  +\int_{(\cD\setminus \cD')^2}  \!\!\!\!\! K(x,y) \mu(\dd x)\mu(\dd y)\\
\le   \delta |\mu|(\cD')^2+\eta.
\end{multline}
The reader can check  that the assumption $\mu\in \cM_K$ implies that $|\mu|(\cD')<\infty$
(and hence $\mu'\in \cM_{K^{(\delta)}}$ which we have used implicitly to obtain \eqref{proxilimit}).
 Next we are going to control the first terms in the r.h.s.\ in   \eqref{ladif1}-\eqref{ladif2}.
 there exists a constant $C$ (allowed to depend on $f$, $\gamma$, and on the covariance kernel $K$) such that for all $\gep$ and $\delta$ sufficiently small
 \begin{equation}\begin{split}\label{proxilimat}
\bbE \left[ \left|  e^{i v(\gep,\theta,\gamma){M}^{(\gamma,\delta)}_{\gep}(f,\go)} 
 -   e^{i v(\gep,\theta,\gamma){M}^{(\gamma)}_{\gep}(f,\go)} \right| \right]
&\le C \sqrt{\delta},\\
  \bbE \left[\left| e^{- {M}^{(2\alpha)}_{0}(e^{|\gamma|^2 L}f^2)}- e^{- {M}^{(2\alpha,\delta)}_{0}(e^{|\gamma|^2 L}f^2)}\right| \right]  &\le  C \sqrt{\delta}.
 \end{split}\end{equation}
Before proving \eqref{proxilimat} let us show how it can be used to conclude
 The combination of \eqref{proxilimit}-\eqref{proxilimat} yields
\begin{multline}
   \limsup_{\gep \to 0} \left|\bbE \left[e^{i \langle X,\mu\rangle + i v(\gep,\theta,\gamma)M^{(\gamma)}_{\gep}(f,\go)} \right]-
 \bbE \left[ e^{-i  \langle X,\mu\rangle-\frac{1}{2}M^{(2\alpha)}_{0}(e^{|\gamma|^2L}f^2)} \right]\right|\\\le 2C \sqrt{\delta}+2 \sqrt{ \delta |\mu|(\cD')^2+\eta}.
\end{multline}
The r.h.s.\ can be made arbitrarily small by taking first $\eta$ small (which fixes $\cD'$) and then $\delta$ small.
which is sufficient to conclude since $\delta$ may be chosen arbitrarily small.

\medskip

\noindent Let us now prove \eqref{proxilimat}. Using the fact that $u\mapsto e^{iu}$ and $u\mapsto e^{-u}$ are Lipshitz function on 
$\bbR$ and $\bbR_+$ respectively  (for the first line we also use that $|\mathfrak{Re}(e^{-i\go}z)|\le |z|$)
we have
\begin{equation}\begin{split}\label{deuzinek}
\left|\bbE \left[ e^{i v(\gep,\theta,\gamma){M}^{(\gamma,\delta)}_{\gep}(f,\go)}- 
e^{i v(\gep,\theta,\gamma){M}^{(\gamma)}_{\gep}(f,\go)} \right]\right|&
\le  v(\gep,\theta,\gamma)\bbE \left[\left|{M}^{(\gamma,\delta)}_{\gep}(f)- 
{M}^{(\gamma)}_{\gep}(f)\right|\right],\\
\left| \bbE \left[ e^{- {M}^{(2\alpha,\delta)}_{0}(e^{L}f^2)}-e^{- {M}^{(2\alpha)}_{0}(e^{L}f^2)}   \right]
\right| &\le \bbE \left[\left| {M}^{(2\alpha,\delta)}_{0}(e^{L}f^2)- {M}^{(2\alpha)}_{0}(e^{L}f^2)\right|\right].
\end{split}\end{equation}
 To bound  the r.h.s.\ of the first line in \eqref{deuzinek}, we rely on Cauchy-Schwartz and evaluate the second moment. Using the independence of $X$ and $Z^{(\delta)}$ and the assumption $(B)$ on $\Delta^{(\delta)}$ setting 

 \begin{equation}
   \Delta^{(\delta)}_{\gep}(x,y):= \int_{(\cD')^2}  \Delta^{(\delta)}(z_1,z_2) \theta_{\gep}(x-z_1)\theta_{\gep}(y-z_2)\dd z_1 \dd z_2.
 \end{equation}
 we obtain that - here we assume that $\gep$ is sufficiently small so that the support of $f$ is included in $\cD'_{\gep}$ (recall \eqref{cdgep})
 \begin{multline}\label{lafoorge}
\bbE\left[ \left|{M}^{(\gamma,\delta)}_{\gep}(f)-{M}^{(\gamma)}_{\gep}(f)\right|^2 \right]
= \int_{(\cD')^2} f(x)f(y) \left(e^{|\gamma|^2 \Delta^{(\delta)}_{\gep}(x,y)}-1\right) e^{|\gamma|^2 K_{\gep}(x,y)}\dd x \dd y \\  \le 
\left( e^{\delta |\gamma|^2}-1 \right) 
\int_{(\cD')^2} f(x)f(y)e^{|\gamma|^2 K_{\gep}(x,y)} \dd x \dd y.
\end{multline}
In the second line above we simply used that  $\Delta^{(\delta)}_{\gep}(x,y)\le \delta$ which follows immediately from the assumption (A) in Lemma \ref{lesasump}.
Using the estimate \eqref{estimao} for $K_{\gep}$ we obtain  
that 

\begin{multline}
\int_{\cD^2} f(x)f(y)e^{|\gamma|^2 K_{\gep}(x,y)} \dd x \dd y
\le   C    \int_{\bbR^d}  f(x)f(y)(|x-y|\vee \gep)^{-|\gamma|^2} \dd x \dd y\\
\le 
e^{C|\gamma|^2} \|f \|^2_2  \int_{\bbR^d} (|z|\vee \gep)^{-|\gamma|^2} \dd z.
\end{multline}
The right-hand side is of order $\gep^{d-|\gamma|^2}$ if $|\gamma|>\sqrt{d}$ and of order $\log (1/\gep)$ if  $|\gamma|=\sqrt{d}$.
This allows to conclude from \eqref{lafoorge} that for a constant $C$ which may depend on all parameters but $\delta$ and $\gep$, we have
\begin{equation}
 v(\gep,\theta,\gamma)^2\bbE \left[\left|{M}^{(\gamma,\delta)}_{\gep}(f)- 
{M}^{(\gamma)}_{\gep}(f)\right|^2\right]\le C \delta.
\end{equation}
Let us now evaluate the r.h.s.\ of the  second line in \eqref{deuzinek}. 
 Since, by Theorem \ref{zabbid}, ${M}^{(2\alpha,\delta)}_{\gep}(e^{L}f^2)$ and ${M}^{(2\alpha)}_{\gep}(e^{L}f^2)$ both converge in $\bbL_1$, 
 it is sufficient to prove a bound on  $$\bbE\left[ \left|{M}^{(2\alpha,\delta)}_{\gep}(e^{|\gamma|^2 L} f^2)-{M}^{(2\alpha)}_{\gep}(e^{|\gamma|^2 L} f^2)\right| \right]$$ which is uniform in $\gep$. Assuming that the support of $f$ is included in $\cD'_{\gep}$ (recall \eqref{cdgep}), letting $Z_{\gep}$ denote the field $Z$ convoluted with $\theta_{\gep}$ we have
 \begin{multline}\label{llhhss}
 \bbE\left[ \left|{M}^{(2\alpha,\delta)}_{\gep}(e^{|\gamma|^2 L} f^2)-{M}^{(2\alpha)}_{\gep}(e^{|\gamma|^2 L} f^2)\right| \right]\\
 \le \bbE \left[ \int_{\cD'} \left| e^{2\alpha Z_{\gep}(x)-2\alpha^2\Delta^{(\delta)}_{\gep}(x)}-1\right| e^{|\gamma|^2 L(x)}f^2(x) M^{(2\alpha)}_\gep(\dd x) \right].
 \end{multline}
 Averaging with respect to $X$, and using the fact that $X$ and $Z$ are independent, we obtain that the l.h.s.\ in
 \eqref{llhhss} is equal to 
 \begin{equation}
  \int_{\cD} \bbE \left[\left| e^{2\alpha Z_{\gep}(x)-2\alpha^2\Delta_{\gep}^{(\delta)}(x)}-1\right| \right] e^{|\gamma|^2 L(x)}f^2(x)\dd x 
  \le \sqrt{ e^{4\alpha^2 \delta}-1} \int_{\cD}  e^{|\gamma|^2 L(x)}f^2(x)\dd x 
\end{equation}
where for the inequality we only used Cauchy-Schwartz together with 
\begin{equation}
 \bbE \left[\left( e^{2\alpha Z_{\gep}(x)-2\alpha^2\Delta_{\gep}^{(\delta)}(x)}-1\right)^2 \right]
 = e^{4 \alpha^2 \Delta_{\gep}^{(\delta)}(x)}-1\le e^{4\alpha^2 \delta}-1. 
\end{equation}

 \qed

\subsection{The martingale approximation of the GMC}\label{martinapprox}
Using the assumption \eqref{ilaunestar}, we can introduce another functional approximation of the field $(X_t)_{t\ge 0}$ besides the convolution approximation.
Introducing the notation
$Q_t(x,y):= \kappa(e^t(x-y)),$ we
 define $(X_t(x))_{x\in \bbR^d, t\ge 0}$ as the Gaussian field parameterized by $\cD\times \bbR_+$ with covariance function
 \begin{equation}\label{labigt}
\bbE[X_s(x)X_t(y)]=K_0(x,y)+\int^{s\wedge t}_0 Q_u(x,y) \dd u=:K_{s\wedge t}(x,y). 
 \end{equation}
 By construction $X_t(\cdot)$ is a martingale for the canonical filtration $(\cF_t)_{t\in [0,\infty]}$ defined by
 \begin{equation}\label{defcft}
 \cF_t:= \sigma( X_s, s\in [0,t)).
 \end{equation}
Note that $\cF_X\subset \cF_{\infty}$, and that the inclusion is strict.
 Since, from our assumptions on $K_0$ and $\kappa$,  $K_t(x,y)$ is  H\"older continuous (in space and time) then by Kolmogorov-Chensov Theorem $(X_t(x))_{x\in \cD, t\ge 0}$ admits a 
modification which is space time continuous.
On the same probability space, one can define a distributional field $X$ indexed by $\cM_K$ by setting 
\begin{equation}\label{defdeX}
 \langle X,\mu \rangle:= \lim_{t\to \infty} \int X_t(x) \dd \mu(x),
 \end{equation}
 (the convergence holds in $\bbL_2$). The field $X$ has covariance $K$ and $(X_t)$ is thus a sequence of approximation for $X$. 
We then consider the  following distribution valued martingale by setting for $f\in C^{\infty}_c(\cD)$
\begin{equation}
 M^{(\gamma)}_t(f):= \int_{\bbR^d} f(x) e^{\gamma X_t(x) -\frac{\gamma^2 K_t(x)}{2}} \dd x.
\end{equation}
The notation $K_t$, $X_t$ and $M^{(\gamma)}_t$ conflicts with $K_\gep$, $X_\gep$ and $M^{(\gamma)}_\gep$ introduced earlier, but we believe this abuse of notation to be harmless for all our purposes: the variable $\gep$ will always that be used for convolution approximation while $t$ and a other latin letters will be used for martingale approximations.

\medskip

The following result illustrates the fact that the martingale approximation has an effect which is similar to the convolution approximation.
It can be deduced from the proof in \cite{Natele} and is also a particular case of \cite[Remark 2.4]{lacoin2020}.
\begin{theorema}\label{zabbid2}
 For any $\alpha \in\bbR$ with $|\alpha |<\sqrt{2d}$, then there exists a random distribution $M^{(\alpha)}_{0}$ such that for every $\theta$  and every continuous $f\in \cC^{\infty}_c(\cD)$, we have the following convergence in $\bbL_1$
  $$   \lim_{t \to \infty}M^{(\alpha)}_{t}(f)= M^{(\alpha)}_{0}(f). $$
 where the distribution $M^{(\alpha)}_{0}$ is the same as in Theorem \ref{zabbid}. 
\end{theorema}

Before proving Proposition \ref{withypo},
we are going to prove that 
$M^{(\gamma)}_t$, once renormalized also converges to a white noise with intensity $M^{(2\alpha)}_0$. We only prove convergence of finite dimensional distribution but, as the reader can check, the proof of tightness (Proposition \ref{tiggh}) needs no adaptation in this case. 
While this result is not necessary to prove Proposition \ref{withypo},  it provides a step of intermediate difficulty and thus serves a didactical purpose. Furthermore Theorem \ref{martinpro} presents an interest in itself since it provides further indication that the scaling limit is universal since it is the same for convolution approximation and for martingale approximation.

\medskip

\noindent  As for the convolution case, we define for $\go\in [0,2\pi)$
\begin{equation}
 M^{(\gamma)}_t(f,\go):= \mathfrak{Re}(e^{-i\go}  M^{(\gamma)}_t(f)).
\end{equation}
and set
\begin{equation}
 \bar v(t,\kappa,\gamma):=\begin{cases}
                           e^{\frac{d-|\gamma|^2}{2}t} \left( \frac 1 2 \int_{\bbR^d} e^{-|\gamma|^2 \ell_{\kappa}(|z|)} \dd z\right)^{-1/2}\quad &\text{ if } |\gamma|>\sqrt{d},\\
                            \left(\frac{\pi^{\frac{d}{2}}t}{\gG(d/2)} \right)^{-1/2} \quad  &\text{ if } |\gamma|=\sqrt{d},
                          \end{cases}
\end{equation}
where 
\begin{equation}
 \ell_{\kappa}(r):= \int^{\infty}_0 \kappa(e^{-t} r)-\kappa(e^{-t})\dd t.
\end{equation}

\begin{theorem}\label{martinpro}
Let $X$ a Gaussian field  on $\cD$ whose covariance kernel satisfies \eqref{ilaunestar}, and $\gamma \in \cP'_{\mathrm{III}}$. We have for every  $f\in C^{\infty}_c(\cD)$
$\go\in [0,2\pi)$, and $\mu \in \cM_K$
\begin{equation}
 \lim_{t \to \infty} \bbE \left[ e^{i\langle \mu,X\rangle+i \bar v(t,\kappa,\gamma)M^{(\gamma)}_{t}(f,\go)} \right]= \bbE \left[ e^{i\langle \mu,X\rangle- \frac{1}{2}M^{(2\alpha)}_{0}(e^{|\gamma|^2 L}f^2)} \right].
\end{equation}
 
\end{theorem}
The function $L$ appearing in the theorem above is simply defined by \eqref{fourme}. Prefer to use $L$ instead of $K_0$ in the result since the latter depends on the specific choice which is made for $\kappa$ while $L$ is entirely entirely determined by the kernel  $K$. Since we have
\begin{equation}
L(x,y):= K_0(x,y) + \int^{\infty}_0 \kappa(e^t|x-y|)\dd t- \log \frac{1}{|x-y|}, 
\end{equation}
looking at the value of the difference of the two last terms near the diagonal, we obtain that $L(x)$ and $K_0(x)$ only differ by a constant
\begin{equation}\label{jkappax}
 L(x,x)= K_0(x,x)-\int^{\infty}_0 \left(1-\kappa(e^{-t})\right) \dd t.
\end{equation}
We introduce the notation
\begin{equation}\label{jkappa}
 j_{\kappa}:= \int^{\infty}_0 \left(1-\kappa(e^{-t})\right) \dd t.
\end{equation}
The proof of Theorem \ref{martinpro} relies on a very simple strategy. 
We need to prove that the quadratic variation of $(M^{(\gamma)}_{t}(f,\go))$ (recall that this a
continuous martingale with respect to the filtration $(\cF_t)$)
satisfies the law of large number given in the following proposition, and then use a martingale Central Limit Theorem which we introduce in the next section.

\begin{proposition}\label{quadraz}
 Under the assumption of Theorem \ref{martinpro} setting $N_t:=M^{(\gamma)}_t(f,\go)$
 we have
 \begin{equation}
  \lim_{t\to \infty} \bar v(t,\kappa,\gamma)^2\langle N\rangle_t =  M^{(2\alpha)}_0(e^{-|\gamma^2| L}f^2).
 \end{equation}
\end{proposition}

\subsection{A martingale CLT with a random variance}

One of the key ideas our proof is the use of the following Central Limit Theorem with a random variance.
We consider $(\cF_t)$ to be a continuous filtration, $(N_t)$ be a continuous local martingale for the filtration $(\cF_t)$, $Z$ a non-negative random variable and 
$v: \bbR_+ \to \bbR$ be a non-increasing positive continuous  function such that $\lim_{t\to \infty} v(t)=0$. The following result is a particular case of \cite[Theorem 5.50, Chap. VIII-Section 5c]{jacodsh} (the reference only treats the case $v(t)=t^{-1/2}$ but this is just a matter of time rescaling).

\begin{theorem}\label{TLC}
 If the quadratic variation of $N$ satisfies the following convergence in probability
 \begin{equation}\label{ziout}
  \lim_{t\to \infty} v(t)^{2} \langle N \rangle_t= Z
 \end{equation}
then we have for every $\xi,\xi'\in \bbR$ and $Y$ bounded real valued $\cF_{\infty}$-measurable random variable 
 \begin{equation}
  \lim_{t\to \infty} \bbE \left[ e^{i\xi v(t) N_t+i\xi' Y} \right]=\bbE\left[ e^{-\frac{\xi^2}{2}Z+i\xi'Y}\right].
 \end{equation} 
In other words we have the following joint convergence in distribution
\begin{equation}
( Y, v(t)N_t)  \ \stackrel{t\to \infty}{\Longrightarrow}  \ (Y,\sqrt{Z}\times \cN)
\end{equation}
where $\cN$ is a standard normal variable independent of $Z$.

\end{theorem}

\begin{proof}[Proof of Theorem \ref{martinpro}]
It follows directly from the combination of Proposition \ref{quadraz} and Theorem \ref{TLC}.

\end{proof}

\subsection{Roadmap to prove Proposition \ref{withypo}}
To prove Proposition \ref{withypo}, we are going to follow some of the same ideas used the proof of Theorem \ref{martinpro}.
We consider a martingale indexed by $t$ whose limit is given by $M^{(\gamma)}_{\gep}$.
For this we consider $(X_{t,\gep})_{x\in \cD_{\gep}}$ (recall \eqref{cdgep}) defined by 
\begin{equation}
 X_{t,\gep}(x)=\int_{\bbR^d} \theta_{\gep}(x-z)X_t(z) \dd z.
\end{equation}
We let  $K_{t,\gep}$ denote the covariance  of $X_{t,\gep}$
\begin{equation}\label{labigtg}
 K_{t,\gep}(x,y):= \int_{(\bbR^d)^2} \theta_{\gep}(x-z_1)\theta_{\gep}(y-z_2) K_t(z_1,z_2)\dd z_1 \dd z_2
\end{equation}
Note that if $X$ is the field defined by \eqref{defdeX} then we have 
$$\lim_{t\to \infty} X_{t,\gep}(x)= X_{\gep}(x) \quad \text{ and } \quad \bbE\left[X_{\gep}(x) \ | \ \cF_t\right]= X_{\gep}(x).$$
Given  $\gamma\in \cP'_{\mathrm{III}}$,   $f\in C^{\infty}_c(\cD)$ and $\go\in [0,2\pi)$, we define
\begin{equation}\label{ngept}
 N^{(\gep)}_{t} = \int_{\cD} f(x) e^{\alpha X_{t,\gep}(x) -\frac{(\alpha^2-\beta^2) K_{t,\gep}(x)}{2}} \cos(\beta X_{t,\gep}(x)-\alpha\beta K_{t,\gep}(x)-\go )\dd x.
\end{equation}
Note that 
\begin{equation}
 N^{(\gep)}_{t}= \bbE[ M^{(\gamma)}_{\gep}(f,\go) \ | \  \cF_t]  \quad \text{ and } \quad  \lim_{t\to \infty}  N^{(\gep)}_{t}= M^{(\gamma)}_{\gep}(f,\go).
 \end{equation}
Instead of  the  asymptotic when of the quadratic variation when $t$ tends to infinity like in Proposition \ref{quadraz}, we must prove a similar statement about $\langle N^{(\gep)}\rangle_\infty$  when $\gep$ tends to  $0$. 
\begin{proposition}\label{quadraz2}
 Under the assumption of Theorem \ref{martinpro} and with $N^{(\gep)}_{t}$
 \begin{equation}
  \lim_{\gep\to 0}  v(\gep,\theta,\gamma)^2\langle N^{(\gep)}\rangle_\infty =  M^{(2\alpha)}_0(e^{-|\gamma^2| L}f^2).
 \end{equation}
\end{proposition}
\noindent We cannot deduce Theorem \ref{martinpro} from Proposition \ref{quadraz2} using Theorem \ref{TLC} but  we will be using similar ideas to conclude.

\subsection{Organization of the paper}\label{lorga}
\begin{itemize}
\item In Section \ref{prilim} we introduce a couple of classical tools and technical estimates that are used throughout the paper.
\item Sections  \ref{sec:quadraz} is devoted to the proof of 
Theorem \ref{martinpro}. 
\item Sections   \ref{sec:geralth2} and \ref{sec:quadraz2} are devoted to the proof of Proposition \ref{withypo}. The proof are a bit more technical than those concerning Theorem \ref{martinpro}, but follow the same main ideas. In Section \ref{sec:geralth2}, we show how Proposition \ref{withypo} can be deduced from Proposition \ref{quadraz2} while in Section 
\ref{sec:quadraz2} we prove Proposition \ref{quadraz2}.
\item Two technical results are proved in appendices.  Lemma \ref{lesasump} 
is proved in Appendix \ref{approxapp}, while Proposition \ref{tiggh} is proved in Appendix \ref{tightness}.
 
\end{itemize}
\medskip

\section{Technical preliminaries}\label{prilim}

\subsection{Gaussian tools}

Before starting the proof, let us mention one inequality and one identity that will be used repeatedly. First,  for any $\sigma>0$ and $t\ge 0$, we have
\begin{equation}\label{gtail}
 \frac{1}{\sqrt{2 \pi} \sigma}\int^{\infty}_t e^{-\frac{u^2}{2\sigma^2 }} \dd u \le e^{-\frac{t^2}{2\sigma}}.
\end{equation}
We refer to this inequality as the Gaussian tail bound.
Our second tool is the Cameron-Martin formula. 
Let $(Y(z))_{z\in \cZ}$ be an arbitrary centered Gaussian field indexed by an arbitrary set $\cZ$. We let  $H$ denote its covariance and $\bbP$ denote its law. 
For any $z\in \cZ$ let us define $\tilde \bbP_z$ the measure tilted by $Y(z)$.
\begin{equation}
 \frac{\dd \tilde \bbP_z}{\dd \bbP}:= e^{Y(z)- \frac{1}{2} H(z,z)}
\end{equation}

\begin{proposition}\label{cameronmartinpro}
Under the probability law $\tilde \bbP_z$, $Y$ is a Gaussian field with covariance $H$,
with mean value equal to
\begin{equation}
 \tilde \bbE_z[ Y(z')]=H(z,z').
\end{equation}
\end{proposition}

\subsection{Estimates for the covariance kernels}

We are going to list a collection of useful estimates for the kernel $K_{\gep}$, $K_t$
$K_{t,\gep}$ defined in \eqref{labig}, \eqref{labigt} and \eqref{labigtg} 
(under the most general assumption \eqref{fourme} for the $K_{\gep}$ and \eqref{ilaunestar} for the others).

\medskip

First note that if we fix a compact in $\cK \subset \cD$, then there exists a constant $C$ (depending on $K$, $\cK$, $\theta$ and $\kappa$) such that for every $t\ge 0$ and every $\gep$ such that $\cD_{\gep}\subset \cK$, every $x$ and $y$ in $\cK$
\begin{equation}\begin{split}\label{estimao}
  \left|K_{t}(x,y)-  \log \left(\frac{1}{|x-y|\vee e^{-t}}\right)    \right|&\le C,\\
\left|K_{\gep}(x,y)-  \log \frac{1}{|x-y|\vee \gep}     \right|&\le C,\\
\left|K_{t,\gep}(x,y)-  \log \frac{1}{|x-y|\vee \gep \vee e^{-t}}   \right|&\le C.
\end{split} 
\end{equation}
The estimates above can be proved by hand using the definitions and are left to the reader.
Secondly, let us give some estimate concerning the local regularity of the Kernel near the diagonal. 

\begin{lemma}\label{localike}
There exists a function $\eta : \bbR_+ \to \bbR_+$ such that $\lim_{a\to 0} \eta(a)=0$ and a constant $C$ which is such that
for every $x,y\in \cK$,  $t\ge 0$ and every $\gep>0$ such that $ \cK\subset \cD_{\gep}$, 

\begin{equation}\begin{split}
|K_{t}(x,y)-K_t(x,x)|&\le \eta(|x-y|)+ C e^{t}|x-y|,\\
|K_{t,\gep}(x,y)-K_{t,\gep}(x,x)|&\le \eta(|x-y|)+ C e^{t}|x-y|.
\end{split}\end{equation}
\end{lemma}

\begin{proof}
We let
 $\eta$ to be the modulus of continuity of  $K_0$ restricted to $\cK$.
 Then we have by definition
$$|K_{0}(x,y)-K_0(x,x)|\le \eta(|x-y|) \text{ and } |K_{0,\gep}(x,y)-K_{0,\gep}(x,x)|\le \eta(|x-y|)$$
 We can thus assume that $K_0\equiv 0$.
 In that case the inequality 
\begin{equation}
 |K_{t}(x,y)-K_{t}(x,x)|\le C e^{t}|x-y|
 \end{equation}
 follows from the fact that $\kappa$ is Lipshitz, and this also holds after convolution.

\end{proof}

\section{Proof of Proposition \ref{quadraz}}\label{sec:quadraz}

Our proof is simply based on an explicit computation of the quadratic variation using Itô calculus.
This computation is  more convenient (for notation) when considering the complex valued martingale
$$W_t:=  M^{(\gamma)}_t(f)=\int_\cD e^{\gamma X_t(x)-\frac{\gamma^2}{2}K_t(x)} f(x) \dd x.$$

\medskip

\noindent The bracket of $M^{(\gamma)}_t(f,\go)$ can then easily be expressed in terms of  $\langle W, \bar W \rangle_t$ and  $\langle W,  W \rangle_t$.

\medskip

\noindent Informally we are going to prove that for large $t$
$$ \dd  \langle W, \bar W \rangle_t \approx C e^{|\gamma|^2-d}M_0^{(2\alpha)}(e^{|\gamma|^2L} f^2) \dd t$$
and that $\dd  \langle W,  W \rangle_t$ is of a much smaller order. 
We then simply integrate that inequality in $t$ to obtain the required asymptotics.

\subsection{Computing the quadratic variation}

We fix $\gamma\in \cP'_{\mathrm{III}}$ and $f\in C^{\infty}_c(\bbR^d)$ be fixed.
We define the complex martingale $W_t$ by
$$W_t:=  M^{(\gamma)}_t(f)=\int_\cD e^{\gamma X_t(x)-\frac{\gamma^2}{2}K_t(x)} f(x) \dd x$$
Now note that for $x\in \cD$, $(X_t(x))_{t\ge 0}$ is a continuous time martingale.
Using Itô calculus we have
\begin{equation}
 \dd W_t:= \gamma\int_\cD e^{\gamma X_t(x)-\frac{\gamma^2}{2}K_t(x)} f(x)  \dd X_t(x) \dd x.
\end{equation}
Now by construction we have $\dd \langle X(x),X(y) \rangle_t= Q_t(x,y)$
and
we obtain the following expressions for the brackets    $\langle W, \bar W \rangle_t$ and     $\langle W,  W \rangle_t$
\begin{equation}\label{integrare}
  \langle W, \bar W \rangle_t =|\gamma|^2\int^t_0 A_s\dd s\quad \text{ and } \quad 
    \langle W,  W \rangle_t =\gamma^2\int^t_0  B_s  \dd s,
\end{equation}
where $A_s$ and $B_s$ are defined by
\begin{equation}
\begin{split}
A_s &:= \int_{\cD^{2}} f(x)f(y)  Q_s(x,y) 
e^{\gamma X_s(x)+\bar \gamma X_s(y) +(\beta^2-\alpha^2)K_s(x) }\dd x \dd y,\\
B_s &:=\int_{\cD^{2}} f(x)f(y) Q_s(x,y) e^{\gamma (X_s(x)+X_s(y))-\gamma^2 K_s(x)}\dd x \dd y.
\end{split}
\end{equation}
The core of our proof is to show that $A_s$ and $B_s$ properly rescaled converge to $M^{(2\alpha)}_0$  and $0$ respectively.
Let us define $\hat K_s=K_s-K_0$ or 
\begin{equation}\label{whops}
\hat K_s(x,y) := \int^{s}_{0} \kappa(e^{t}|x-y|)\dd t
\end{equation}
 and set 
$$a(s,\kappa,\gamma):=\int_{\bbR^{2d}}  Q_s(0,z) e^{|\gamma|^2 \hat K_s(0,z)} \dd z.$$

\begin{proposition}\label{lealeb}
 
We have 
 \begin{equation}\begin{split}\label{grandz}
  \lim_{t\to \infty} a(t,\kappa,\gamma)^{-1} A_t&= M^{(2\alpha)}_0(e^{|\gamma^2|K_0}f^2),\\
    \lim_{t\to \infty} a(t,\kappa,\gamma)^{-1} B_t&= 0.
 \end{split}\end{equation}

\end{proposition}

The proof of Proposition \ref{lealeb} is detailed in the next subsection and requires several technical steps.
Let us show first that it implies Proposition \ref{quadraz}.

\begin{proof}[Proof of Proposition \ref{quadraz} using Proposition \ref{lealeb}]

First let us show that 
\begin{equation}\label{xlink}
 \lim_{t\to \infty}\bar v (t,\kappa,\gamma)^2  |\gamma|^2 \int^t _0 a(s,\kappa,\gamma)\dd s  = 2e^{-|\gamma^2| j_{\kappa}}. 
\end{equation}
 We  have 
\begin{multline}\label{scalion}
 \int^t _0 |\gamma|^2 a(s,\kappa,\gamma)\dd s= \int_{|z|\le 1} \left(e^{|\gamma|^2\hat K_t(0,z)} -1\right)\dd z\\
 = e^{(|\gamma|^2-d)t} \int_{|y|\le e^t} \left(e^{-|\gamma|^2 \int^t_0 (1-\kappa(e^{-s}y))\dd s} -e^{-|\gamma|^2 t} \right)\dd y,
\end{multline}
where the last inequality is obtained using the change of variable $y=e^{t}z$.
 When $|\gamma|>\sqrt{d}$ the $e^{-|\gamma|^2 t}$ term, once integrated tend to zero.
Now  we have  
\begin{equation}\label{convi}
\lim_{t\to \infty }\int^t_0 (1-\kappa(e^{-s}y))\dd s=\ell_{\kappa}(y)+j_{\kappa},
\end{equation}
and when  $|y|\le e^t$ we have 
\begin{equation}\label{domini}
 \int^t_0 (1-\kappa(e^{-s}y))\dd s \ge    (\log|y|-C)_+   
\end{equation}
With \eqref{convi} and \eqref{domini} we use dominated convergence and obtain 
\begin{equation}
 \lim_{t \to \infty }\int_{\bbR^d} e^{-|\gamma|^2 \int^t_0 (1-\kappa(e^{-s}|y|)\dd s} \ind_{\{|y|\le e^{t}\}} \dd y
 = e^{-|\gamma^2| j_{\kappa}}\int_{\bbR^d} e^{-|\gamma|^2  \ell_\kappa(y)} \dd y,
\end{equation}
so that \eqref{xlink} holds.
When $|\gamma|=\sqrt{d}$, considering the first equality in \eqref{scalion}, we disregard the $-1$ in the integral which only yields a contribution of constant order. Observe that we have  $\hat K_t(0,z)\le t$ for all $z$, and that whenever $|z|\in [ e^{-t},1]$ 
\begin{equation}
 \hat K_t(0,z)= \log \frac{1}{|z|}+ \int^{\log (1/|z|)}_0 (\kappa(e^t|z|)-1)\dd z=  \log \frac{1}{|z|}-j_{\kappa}
\end{equation}
This yields 
\begin{equation}
\int_{|z|\le 1} e^{d\hat K_t(0,z)} \dd z
= e^{-j_{\kappa}}\int_{|z|\in [ e^{-t},1]}  |z|^{-d} \dd z+ \int_{|z|\le e^{-t}} e^{d\hat K_t(0,z)}\dd z.
\end{equation}
The first term is exactly $e^{-j_{\kappa}} t \frac{ 2 \pi^{d/2}}{\gG(d/2)}$
(the last factor being the volume of the $d-1$ dimensional sphere) while the other term is of order one
which yield \eqref{xlink} in that case too.

\medskip

\noindent The quadratic variation of $N_t=M^{(\gamma)}_t(f,\go)$ is a linear combination of 
  $\langle W, \bar W \rangle_t$ and $\langle W,  W \rangle_t$
 we have
 \begin{equation}\label{lerougelenoir}
  \langle N\rangle_t= \frac{1}{2}\langle W, \bar W \rangle_t + \frac{1}{2} \mathfrak{Re}(e^{-2i\go}\langle W,  W \rangle_t)
 \end{equation}
In view of \eqref{xlink}-\eqref{lerougelenoir}  and \eqref{jkappax}, we only need to prove the following  convergence in $L_1$
\begin{equation}\label{trucaprouver}
  \lim_{t \to \infty } \frac{\langle W, \bar W \rangle_t}{\int^t_0 |\gamma|^2 a(s,\kappa,\gamma) \dd s}= M^{(2\alpha)}_0(e^{|\gamma^2|K_0} f^2) \quad \text{ and } \quad  \lim_{t \to \infty } \frac{\langle W,  W \rangle_t}{\int^t_0  a(s,\kappa,\gamma) \dd s}=0,
\end{equation}
As a direct consequence of \eqref{integrare} and \eqref{grandz} we have
\begin{equation}\label{trucabis}
 \lim_{t \to \infty } \frac{\langle W, \bar W \rangle_{[\sqrt{t},t]}}{\int^t_{\sqrt{t}} |\gamma|^2 a(s,\kappa,\gamma) \dd s}= M^{(2\alpha)}_0(e^{|\gamma^2|K_0} f^2) \quad \text{ and } \quad  \lim_{t \to \infty } \frac{\langle W,  W \rangle_{[\sqrt{t},t]}}{\int^t_{\sqrt{t}}  a(s,\kappa,\gamma) \dd s}=0.
\end{equation}
To conclude we only need to check that both $\int^{\sqrt{t}}_{0}  a(s,\kappa,\gamma) \dd s$ and
$\langle W, \bar W \rangle_{\sqrt{t}}$ are negligible with respect to $\int^t_{0} |\gamma|^2 a(s,\kappa,\gamma) \dd s$.
For the first quantity, it is sufficient to observe that from definition and \eqref{estimao},
 $a(s,\kappa,\gamma)$ is of order $e^{(|\gamma|^2-d)s}$.
As for the second, we have
 \begin{multline}
\bbE \left[ \langle W, \bar W \rangle_u \right]
=\bbE \left[|W_u-W_0|^2\right]
\le  \int_{\cD^2} f(x)f(y) e^{|\gamma|^2 K_u(x,y)} \dd x \dd y\\ 
\le  e^{|\gamma|^2\| K_0\|_{\infty}} \| f\|^2_2  \int_{\bbR^d}  e^{|\gamma|^2\hat K_u(0,z)} \dd z,
\end{multline}
Now  $\int_{\bbR^d}  e^{|\gamma|^2\hat K_u(0,z)} \dd z$ is either of order $u$ (if $|\gamma|^2=d$) or $e^{(|\gamma|^2-d)u}$ (if $|\gamma|^2>d$) and with $u=\sqrt{t}$ this is negligible w.r.t.\
 $\int^t_{0} |\gamma|^2 a(s,\kappa,\gamma) \dd s$ in all cases.

\end{proof}

\subsection{Proof of Proposition \ref{lealeb}}

For the sake of simplicity we are going to assume (recall \eqref{ilaunestar}) that $K_0\equiv 0$. This does not alter the proof at all but provides some welcome simplification in the notation (for instance we have $K_t(x):=K_t(x,x)=t$ for every $x$). 
Since in that case $K$ is translation invariant, we can extend it to a kernel $\bbR^d\times \bbR^d$, thus without loss of generality we are going to assume that $\cD=\bbR^d$.
In this case note that $L(x,x)= -j_{\kappa}$  for all $x\in \bbR^d$. 
To alleviate further  the notation we write $a(t)$ for $a(t,\kappa,\gamma)$.
We set for this proof $u:=u_t= t-\log t$.
We are going to show the result using intermediate steps. Recalling the notation \eqref{intervalnotation} we set
\begin{equation}\begin{split}
       A^{(1)}_t&:=
 \int_{\bbR^{2d}} f^2(x)  Q_t(x,y) e^{\gamma X_t(x)+\bar \gamma X_t(y) +(\beta^2-\alpha^2)t }\dd x \dd y,\\
              A^{(2)}_t&:= \int_{\bbR^{2d}} f^2(x)  Q_t(x,y) e^{\gamma X_u(x)+\bar \gamma X_u(y)+|\gamma|^2 K_{[u,t]}(x,y) +(\beta^2-\alpha^2)u }\dd x \dd y \\
              &=   \bbE[A^{(1)}_t \ | \ \cF_u],
                \end{split}
\end{equation}
Using the triangle inequality we have
\begin{multline}\label{4term}
 |A_t-a(t)M^{(2\alpha)}_0(f^2)|\le  |A_t-A^{(1)}_t|\\+ |A^{(1)}_t-A^{(2)}_t|+ |A^{(2)}_t- a(t)M^{2\alpha}_u(f^2)|
 + a(t)|M^{(2\alpha)}_u(f^2)-M^{(2\alpha)}_0(f^2)|.
\end{multline}
We are going to show that the expectation of each of the terms of the right-hand side are 
$o(e^{(|\gamma|^2-d)t})$ (which is the same as $o(a(t)))$. 
The fourth term is controlled by applying Theorem \ref{zabbid2}.
The first three terms require more work.

\subsubsection*{Step 1: Bounding $\bbE[|A_t-A^{(1)}_t|]$}

We let $w_f(u):= \max_{|x_1-x_2|\le u} f(x_1-x_2)$ denote the modulus of continuity of $f$.
We have
\begin{equation}
 |A_t-A^{(1)}_t|\le  \int_{\bbR^d\times \bbR^d} |f(x)| |f(y)-f(x)|  Q_t(x,y) e^{\alpha (X_t(x)+X_t(y)) +(\beta^2-\alpha^2)t }\dd x \dd y
\end{equation}
and hence 
\begin{multline}
 \bbE\left[ |A_t-A^{(1)}_t| \right] \le w_f(e^{-t})\int_{\bbR^d} |f(x)|    \ind_{\{|y-x|\le e^{-t}\}} e^{|\gamma|^2 t}   \dd x \dd y 
 \\ = \frac{\pi^{d/2}}{\Gamma\left(\frac{d}{2}+1\right)}  e^{(|\gamma|^2-d)t} w_f(e^{-t})\int_{\bbR^d} |f(x)| \dd x.
\end{multline}
This implies that 
\begin{equation}
 \lim_{t \to \infty}   e^{(d-|\gamma|^2)t}  \bbE\left[ |A_t-A^{(1)}_t| \right]=0.
\end{equation}

\subsubsection*{Step 1: Bounding $\bbE[|A^{(1)}_t-A^{(2)}_t|]$}
Now let us consider the second term $A^{(2)}_t-A^{(1)}_t$. This is the most delicate step.
Let us set 
\begin{equation}\begin{split}\label{fdsre}
\xi_t(x,y)&:=  e^{\gamma X_t(x)+\bar \gamma X_t(y) +(\beta^2-\alpha^2)t} - \bbE[e^{\gamma X_t(x)+\bar \gamma X_t(y) +(\beta^2-\alpha^2)t} \ | \ \cF_u],\\ 
\zeta_t(x,y)&:=Q_t(x,y)f(x)^2\xi_t(x,y)
\end{split}\end{equation}
we have 
\begin{equation}\label{thezera}
 A^{(1)}_t-A^{(2)}_t=\int_{\bbR^{2d}}\zeta_t(x,y) \dd x \dd y
 \end{equation}
Our method to bound $A^{(1)}_t-A^{(2)}_t$ depends on whether $\alpha\in [0,\sqrt{d}/2)$ or $\alpha\in 
[\sqrt{d}/2,\sqrt{d/2})$ (by symmetry we can assume without loss of generality that $\alpha\ge 0$).
When $\alpha<\sqrt{d}/2$ it is sufficient to compute the second moment of  $A^{(2)}_t-A^{(1)}_t$ 
We have 
\begin{equation}
  \bbE \left[ |A^{(1)}_t-A^{(2)}_t|^2\right]=\int_{\bbR^{4d}} \bbE \left[ \zeta_t(x_1,y_1)\bar \zeta_t(x_2,y_2)\right]\dd x_1 \dd x_2\dd y_1 \dd y_2.
\end{equation}
We have for any $x_1,x_2,y_1,y_2$
\begin{equation}
 |\bbE \left[ \xi_t(x_1,y_1)\bar \xi_t(x_2,y_2)\right]|\le \frac 1 2 \left( 
  \bbE \left[ |\xi_t(x_1,y_1)|^2  \right] + \bbE \left[|\bar \xi_t(x_2,y_2)|^2\right]\right)
  \le e^{(6\alpha^2+2\beta^2)t}.
\end{equation}
Now by construction, whenever $|x_1-x_2|\ge  e^{-u}+2 e^{-t},$ we have
\begin{equation}\label{ztruc}
 \bbE[\zeta_t(x_1,y_1)\bar \zeta_t(x_2,y_2) \ | \ \cF_u]=0,
\end{equation}
since at distance $e^{-u}$ the correlation of the field $X_{[u,t]}$ vanish and the prefactor $Q_t(x_1,y_1)Q_t(x_2,y_2)$ is zero if $|x_1-y_1|\vee |x_2-y_2| \ge e^{-t}$.
Hence we have 
\begin{multline}
   \bbE \left[ |A^{(2)}_t-A^{(1)}_t|^2\right] \\ \le \int_{\bbR^{4d}} e^{(6\alpha^2+2\beta^2)t} f(x_1)^2f(x_2)^2Q_t(x_1,y_1)Q_t(x_2,y_2)
   \ind_{\{|x_1-x_2|\le 3e^{-u}\}} \dd x_1 \dd x_2 
   \dd y_1 \dd y_2 \\ 
   \le C e^{2(|\gamma|^2-d)t} e^{4\alpha^2 t -du}     \| f\|^4_4.
\end{multline}
If $4\alpha^2<d$ then the second exponential factor goes to zero (recall that $u=t-\log t$) and we can deduce that 
\begin{equation}
 \lim_{t\to \infty}   e^{(d-|\gamma|^2)t}  \bbE \left[ |A^{(2)}_t-A^{(1)}_t|\right] \le  \lim_{t\to \infty}    e^{(d-|\gamma|^2)t}  \bbE \left[ |A^{(2)}_t-A^{(1)}_t|^2\right]^{1/2}=0.
\end{equation}
When $\alpha\in [ \sqrt{d}/2,\sqrt{d/2})$ we need to combine the above argument with a truncation proceedure.
 Let us fix some parameter $\gl$ that satisfies the following assumptions
\begin{equation}\label{lambdabal}
 \gl\in (2\alpha,4\alpha) \quad \text{ and } \quad  \frac{1}{2}(4\alpha - \gl)^2> 4\alpha^2- d.
\end{equation}
Such a value of $\gl$ exists when $\alpha \in [ \sqrt{d}/2,\sqrt{d/2})$.
We introduce the event
\begin{equation}\label{defcat}
 \cA_t(x):=\{ X_u(x)\le \gl  u \}  
\end{equation}
Nowe recalling \eqref{thezera} 
we have 
\begin{equation}\label{laladeco}\begin{split}
 A^{(2)}_t-A^{(1)}_t&=\int_{\bbR^{2d}}\zeta_t(x,y) \ind_{\cA^{\cc}_t(x)} \dd x \dd y + \int_{\bbR^{2d}}\zeta_t(x,y) \ind_{\cA_t(x)} \dd x \dd y
\\& =: W^{(1)}_t+W^{(2)}_t.
\end{split}\end{equation}
We are going to compute the first moment of $W^{(1)}_t$ and the second moment of $ W^{(2)}_t$. We have 
\begin{multline}\label{zitz}
  \bbE[|\zeta_t(x,y)| \ind_{\cA^{\cc}_t(x)} ]\le  Q_t(x,y)|f(x)|^2 \bigg[ e^{(\beta^2-\alpha^2)t}
  \bbE \left[ e^{\alpha \left(X_t(x)+ X_t(y)\right)} \ind_{\cA^{\cc}_t(x)}\right] \\
  +
   e^{(\beta^2-\alpha^2)u+ |\gamma|^2K_{[u,t]}(x,y)} \bbE \left[ e^{\alpha \left(X_u(x)+ X_u(y)\right)} \ind_{\cA^{\cc}_t(x)}\right]\Bigg].
\end{multline}
Using the Cameron-Martin formula (Proposition \ref{cameronmartinpro}) to compute the expectations we find
\begin{equation}\begin{split}
 \bbE \left[ e^{\alpha \left(X_t(x)+ X_t(y)\right)} \ind_{\cA^{\cc}_t(x)}\right]&= e^{\alpha^2(t+K_t(x,y))} \bbP\left[ X_u(x)>
 \gl u- \alpha(u+K_u(x,y))\right],\\
  \bbE \left[ e^{\alpha \left(X_u(x)+ X_u(y)\right)} \ind_{\cA^{\cc}_t(x)}\right]&= e^{\alpha^2(u+K_u(x,y))} \bbP\left[ X_u(x)>
 \gl u- \alpha(u+K_u(x,y))\right].
 \end{split}
 \end{equation}
Using $u$ and $t$ as upper bounds for $K_t$ and $K_u$ we obtain that 
\begin{equation}
   \bbE[|\zeta_t(x,y) \ind_{\cA^{\cc}_t(x)} |]\le 2  Q_t(x,y)|f(x)|^2 e^{|\gamma|^2 t} \bbP[X_u(x)\ge (\gl-2\alpha) u],
\end{equation}
Recalling that $\gl>2\alpha$, using Gaussian tail estimates \eqref{gtail}, we bound the above probability by $e^{-\frac{(\gl-2\alpha)^2t}{2}}$, and obtain
\begin{equation}\label{W11}
 \bbE \left[|W^{(1)}_t|\right] \le 2 \|f \|^2_2  e^{t|\gamma|^2-\frac{(\gl-2\alpha)^2t}{2}}  \int_{\bbR^d} \kappa( e^t |z|)\dd z 
 \le C \|f \|^2_2 e^{t(|\gamma|^2-d)-\frac{(\gl-2\alpha)^2t}{2}}.
\end{equation}
Let us now control the second moment of $W^{(2)}_t$.
We have
\begin{equation}
 \bbE\left[|W^{(2)}_t|^2 \right]=\int_{\bbR^{4d}} \bbE\left[\zeta_t(x_1,y_1)\bar \zeta_t(x_2,y_2)\ind_{\cA_t(x_1)\cap \cA_t(x_2)}\right]\dd x_1 \dd x_2 \dd y_1 \dd y_2
\end{equation}
 From \eqref{ztruc}, 
 as $\cA_t(x_1)\cap \cA_t(x_2)$ is $\cF_u$ measurable,  since  the correlations of the increments of $X_{[u,t]}$  have range at most $e^{-u}$, and 
we have 
\begin{equation}
\bbE\left[\zeta_t(x_1,y_1)\bar \zeta_t(x_2,y_2)\ind_{\cA_t(x_1)\cap \cA_t(x_2)}\right]=0 \quad \text{ if } \quad  |x_1-x_2|\le  3e^{-u}
\end{equation}
When $|x_1-x_2|<  3e^{-u}$ we have
\begin{multline}
\left|\bbE\left[\zeta(x_1,y_1)\bar \zeta(x_2,y_2)\ind_{\cA_t(x_1)\cap \cA_t(x_2)}\right]\right|\\ \le 
|f(x_1)f(x_2)|^2 Q_t(x_1,y_1)Q_t(x_2,y_2) \frac{1}{2}
\bbE\left[ |\xi_t(x_1,y_1) |^2\ind_{\cA_t(x_1)}+ |\xi_t(x_2,y_2) |^2 \ind_{\cA_t(x_2)}\right],
\end{multline}
so that 
\begin{multline}\label{wric}
  \bbE\left[|W^{(2)}_t|^2 \right] \le  \int_{\bbR^{4d}}|f(x_1)f(x_2)|^2 Q_t(x_1,y_1)Q_t(x_2,y_2)
  \ind_{\{|x_1-x_2|\le 3e^{-u}\}} \dd x_1 \dd x_2 \dd y_1 \dd y_2
  \\ \times \sup_{|x-y|\le e^{-t}} \bbE \left[  |\xi_t(x,y) |^2\ind_{\cA_t(x)} \right].
\end{multline}
The integral in the first line is smaller than  $C \|f\|^4_4 e^{-d(2t+u)} $.
Finally, using the Cameron-Martin formula (Proposition \ref{cameronmartinpro}), we have 
\begin{multline}\label{wrac}
 \bbE \left[ |\xi_t(x,y) |^2\ind_{\cA_t(x)}\right]\le \bbE \left[ \left|e^{\gamma X_t(x)+\bar \gamma X_t(y) +(\beta^2-\alpha^2)t}\right|^2\ind_{\cA_t(x)}\right]
 \\ =  e^{4\alpha^2 K_t(x,y)+ |\gamma|^2 t}   \bbP\left[ X_u(x)\le \gl u- 2\alpha(u+K_u(x,y))  \right]
\end{multline}
Using the fact that $K_u(x,y)\ge u-1$ if $|x-y|\le e^{-t}$, and the fact that $\gl<4\alpha$
\eqref{lambdabal}, we have from Gaussian tails estimates \eqref{gtail}
\begin{equation}\label{wroc}
  \bbP\left[ X_u(x)\le \gl u- 2\alpha(u+K_u(x,y))  \right]\le C e^{-\frac{(4\alpha-\gl)^2 u}{2}}
\end{equation}
Combining \eqref{wric}-\eqref{wrac}-\eqref{wroc} yields 
\begin{equation}\label{W22}
  \bbE\left[|W^{(2)}_t|^2 \right] \le C \|f\|^4_4 e^{ [4\alpha^2+2(|\gamma|^2-d)]t-\left(\frac{(4\alpha-\gl)^2}{2}+d\right) u  }
  \le \|f\|^4_4 e^{ 2(|\gamma|^2 -d-\delta)t}
\end{equation}
where, using the fact that $|t-u|\ll t$, the last line is valid for $t$ sufficiently large  with (cf. \eqref{lambdabal})
\begin{equation}
 \delta:=\frac{1}{4}\left[ (4\alpha - \gl)^2-4\alpha^2+ d\right]>0.
\end{equation}
Recalling \eqref{laladeco}, combining \eqref{W11} and \eqref{W22} we obtain that 
\begin{equation}
 \lim_{t\to \infty} e^{ (d-|\gamma|^2)t} \bbE \left[|A^{(2)}-A^{(1)}| \right]=0.
\end{equation}

\subsubsection*{Step 3: Bounding $\bbE[|A^{(2)}- a(t)M^{(2\alpha)}_0(f^2)|]$ }
This is easier and just comes down to a simple first moment estimate. We have
\begin{equation}
a(t)M^{(2\alpha)}_u(f^2 ):= \int_{\bbR^{2d}} f(x)^2 Q_t(x,y) e^{|\gamma|^2 K_t(x,y)} e^{2\alpha X_u(x)-2\alpha^2 u} \dd x \dd y
\end{equation}
so that from Jensen's inequality
\begin{multline}\label{leyield}
\bbE \left[|a(t)M^{2\alpha}_u(f^2 )-A^{(2)}_t| \right]\le \int_{\bbR^d}  f(x)^2 Q_t(x,y)
e^{|\gamma|^2 K_t(x,y)}  \\
\times
\bbE\left[ \left|e^{2\alpha X_u(x)-2\alpha^2 u} - e^{\gamma X_u(x)+\bar \gamma X_u(y)-|\gamma|^2 K_u(x,y)+(\beta^2-\alpha^2)u}
 \right| \right]\dd x \dd y.
\end{multline}
Using the Cameron-Martin formula (Proposition \ref{cameronmartinpro}) the expectation in the integral is equal to 
\begin{multline}\label{wert}
 \bbE\left[e^{2\alpha X_u(x)-2\alpha^2 u} \left| 1 - e^{\bar \gamma (X_u(y)-X_u(x))+|\gamma|^2(u-K_u(x,y))}
 \right| \right]
 \\= \bbE\left[\left| 1 - e^{\bar \gamma (X_u(y)-X_u(x))-\bar \gamma^2 (u-K_u(x,y))}
 \right| \right].
\end{multline}
We can explicitly compute the second moment in the last line and obtain
\begin{equation}
 \bbE\left[\left| 1 - e^{\bar \gamma (X_u(y)-X_u(x))-\bar \gamma^2 (u-K_u(x,y))}
 \right|^2 \right]=e^{2|\gamma|^2 (u-K_u(x,y))}-1,
\end{equation}
Now when $|x-y|\le e^{-t}$,  using Lemma \ref{localike} \footnote{the  term $\eta(e^{-t})$ is not needed when $K_0\equiv 0$, cf. the proof of Lemma \ref{localike} but adding it make the proof easier to adapt to the general case.}(note that $K_u(x,x)=u$)
\begin{equation}\label{zert}
 e^{2|\gamma|^2 (u-K_u(x,y))}-1 \le  e^{ |\gamma|^2(C e^{u-t}+\eta(e^{-t}))}-1 \le  \delta(t).
\end{equation}
where $\lim_{t\to \infty}\delta(t)=0$ 
Going back to \eqref{leyield}, combining \eqref{wert}-\eqref{zert} yields 
\begin{multline}
\bbE \left[|a(t)M^{2\alpha}_u(f^2 )-A^{(2)}_t| \right] \\ \le   C \sqrt{ \delta(t)}\int_{\bbR^{2d}}  f(x)^2 Q_t(x,y)
e^{|\gamma|^2 K_t(x,y)}  \dd x \dd y= C \sqrt{ \delta(t)}\|f\|^2_2 a(t),
\end{multline}
and finishes the proof for the convergence of $A_t$. We are left with that of $B_t$.

\subsubsection*{Final step: Bounding $\bbE[ |B_t|]$}

For $B_t$ we use the same ideas as for $A_t$, but require only one intermediate step.
We define 
\begin{equation}
B^{(1)}_t= \bbE [B_t \ | \ \cF_u ]= \int_{\bbR^{2d}} f(x)f(y) Q_s(x,y) e^{\gamma (X_u(x)+X_u(y))-\gamma^2 (u-K_{[u,t]}(x,y))}\dd x \dd y,
\end{equation}
a give a separate bound for $ \bbE \left[|B^{(1)}_t|\right]$ and $\bbE \left[|B^{(1)}_t-B_t|\right]$.
We have 
\begin{equation}\begin{split}
 \bbE \left[|B^{(1)}_t|\right]&\le \int_{\bbR^{2d}} f(x)f(y) Q_t(x,y) \bbE\left[ e^{\alpha (X_u(x)+X_u(y))+ (\beta^2-\alpha^2) (u - K_{[u,t]}(x,y))}\right]\dd x \dd y\\
 &\le \int_{\bbR^{2d}} f(x)f(y) Q_s(x,y)  e^{\alpha^2 K_t(x,y)+ \beta^2(u-K_{[u,t]}(x,y))
 }\dd x \dd y\\
 &\le C \|f\|^2_2 e^{(|\gamma|^2-d)t - 2\beta^2(t-u)}.
\end{split}\end{equation}
and hence 
$\lim_{t\to \infty} e^{d-|\gamma|^2}\bbE \left[|B^{(1)}_t|\right]=0$.
We can then bound $\bbE[ |B_t-B^{(1)}_t|]$. The reasoning is identical to that used for 
$A^{(1)}-A^{(2)}_t$ so we provide details only  the more delicate case $\alpha\in [\sqrt{d}/2,\sqrt{d/2})$. We define $\cA_t(x)$ as in \eqref{defcat}
and in analogy with \eqref{fdsre} set
\begin{equation}\begin{split}
\xi'_t(x,y)&:=  e^{\gamma (X_t(x)+ X_t(y)) - \gamma^2 t} - \bbE[e^{\gamma X_t(x)+ \gamma X_t(y)  - \gamma^2 t} \ | \ \cF_u],\\ 
\zeta'_t(x,y)&:=Q_t(x,y)f(x)f(y)\xi_t(x,y).
\end{split}\end{equation}
We have
\begin{equation}\label{tyu}
 |B^{(1)}_t-B_t|\le 
\left| \int_{\bbR^{2d}} \zeta'_t(x,y)\ind_{\cA^{\cc}_t(x)}\dd x \dd y \right|+ \left| \int_{\bbR^{2d}} \zeta'_t(x,y)\ind_{\cA_t(x)}\dd x \dd y\right|
\end{equation}
We prove similarly to \eqref{zitz} that 
\begin{equation}
  \bbE \left[|\zeta'_t(x,y)|\ind_{\cA^{\cc}_t(x)}\right]\le Q_t(x,y)f(x)f(y) e^{|\gamma|^2 t- \frac{(\gl-2\alpha)^2 u}{2}},
\end{equation}
which implies that
\begin{equation}
\bbE \left[ \left| \int_{\bbR^{2d}} \zeta'_t(x,y)\ind_{\cA^{\cc}_t(x)}\dd x \dd y\right| \right] \le C \| f\|_2^2  e^{(|\gamma|^2-d) t- \frac{(\gl-2\alpha)^2 u}{2}}.
\end{equation}
Finally we control the second moment of the first term in \eqref{tyu} like that of $W^{(2)}_t$ in \eqref{wric}. Repeating the same steps we prove that 
\begin{equation}
\bbE \left[ \left| \int_{\bbR^{2d}} \zeta'_t(x,y)\ind_{\cA_t(x)}\dd x \dd y \right|^2\right]
\le C \|f \|^4_4 e^{ [4\alpha^2+2(|\gamma|^2-d)]t-\left(\frac{(4\alpha-\gl)^2}{2}+d\right) u  }
\end{equation}
and conclude in the same manner.

\qed

\section{Deducing Proposition \ref{withypo} from Proposition \ref{quadraz2} } \label{sec:geralth2}

The present section and the next are dedicated to the proof of Proposition \ref{withypo} on which the main result of the paper, Theorem \ref{main} relies.
While with the combination of martingale filtration and convolution, the notation become more cumbersome and the computations a bit more delicate, the proof rely almost exclusively on the ideas developped in Section \ref{sec:quadraz} and on an adaptation of Theorem \ref{TLC}.

\medskip

In this section, we show that Proposition \ref{quadraz2} implies  Proposition \ref{withypo}. 
Recall that $N^{(\gep)}_t= \bbE[M^{(\gamma)}_{\gep}(f,\go) \ | \ \cF_t]$.
We set for the proof 
$Z:= M^{(2\alpha)}(e^{L}f^2)$ and $Z^A_t:=\bbE[ Z\wedge A \ | \ \cF_t]$. We write simply $v(\gep)$ for $v(\gep,\theta,\gamma)$.
Since $N^{(\gep)}_\infty=M^{(\gamma)}_{\gep}(f,\go)$ we need to show that 
\begin{equation}\label{tiziro}
 \lim_{\gep \to 0}
\bbE\left[e^{i\langle X,\mu\rangle + i v(\gep)N^{(\gep)}_\infty}-e^{e^{i\langle X,\mu\rangle}-\frac{1}{2}Z}\right]=0.
\end{equation}
We set $t=t_{\gep}:= \sqrt{\log 1/\gep}$,
and (with the convention $\inf \emptyset =\infty$)
\begin{equation}
 T(A,\gep):= \inf\{ s\ge 0 \ : \ \langle N^{(\gep)}\rangle_s \ge v(\gep)^{-2} A \}.
\end{equation}
We also set 
\begin{equation}
 Y:= e^{i\langle X,\mu\rangle} \quad \text{ and } \quad   Y_t:= \bbE[ Y \ | \ \cF_t]
\end{equation}
Then we decompose the expectation we want to bound, recalling \eqref{intervalnotation}
\begin{multline}
 \left|\bbE\left[Y(e^{i v(\gep)N^{(\gep)}_\infty}-e^{-\frac{1}{2}Z})\right]\right|\\
 \le \left|\bbE\left[  Ye^{i v(\gep)N^{(\gep)}_{\infty}}- Y_{t} e^{i v(\gep)N^{(\gep)}_{[t\wedge  T, T]}}\right]\right|+\left| \bbE\left[ Y_t \left(e^{i v(\gep)N_{[t\wedge T, T]}} -e^{-\frac{1}{2} Z^A_t}\right)\right]\right| \\ +\left| \bbE\left[  Y_te^{-\frac{1}{2}Z^A_t}- Y e^{-\frac{1}{2}Z}\right]\right|
 \end{multline}
Denoting the three terms in the r.h.s.\ by $E_1(\gep,A)$, $E_2(\gep,A)$ and $E_3(\gep,A)$,
we are going to prove \eqref{tiziro} by showing that 
 \begin{equation}\label{ze123}
 \forall j \in \{1,2,3\}, \quad \lim_{A\to \infty} \limsup_{\gep \to 0} E_j(\gep,A)=0.
 \end{equation}
 For $E_3$ we can use dominated convergence twice 
 $$\lim_{A\to \infty}\lim_{\gep\to 0}E_3(\gep,A)=\lim_{A\to \infty}\bbE\left[  Y(e^{-\frac{1}{2}Z\wedge A}-e^{-\frac{1}{2}Z})\right]=0.$$
For $E_1$ we have (recall that $|Y|,|Y_t|\le 1$)
 \begin{equation}\label{yena2}
\bbE\left[    \left|Ye^{i v(\gep)N^{(\gep)}_{\infty}}- Y_t e^{i v(\gep)N^{(\gep)}_{[t\wedge  T, T]}}\right|\right]
\le   2 \bbP[T<\infty]+ \bbE\left[ \left|Y- Y_te^{-iv(\gep)N^{(\gep)}_{t}}\right|\right]
 \end{equation}
Combining Proposition \ref{quadraz} and Portmanteau Theorem we have
\begin{equation}
\limsup_{\gep \to 0} \bbP[T(A,\gep)<\infty]\le \bbP[ Z\ge A].
 \end{equation}
To bound the second term in \eqref{yena2}, since $Y_t$ converges to $Y$ in $\bbL_1$, we only need to show that $v(\gep)N^{(\gep)}_{t}$ tend to zero in probability.
 The reader can check that with our choice for $t_{\gep}$
 \begin{equation}\label{sgonmom}
  \bbE \left[ (N^{(\gep)}_{t})^2\right]\le \begin{cases}
                                             Ce^{(|\gamma|^2-d)t} \quad &\text{ if } |\gamma|^2>d,\\
                                             C t \quad &\text{ if } |\gamma|^2=d,
                                           \end{cases}
 \end{equation}
 and the r.h.s.\ in \eqref{sgonmom} much smaller than $v(\gep)^{-2}$ in both cases.
 Finally for $j=2$, using the fact that $(N^{(\gep)}_{[t\wedge T, s\wedge T]})_{s\ge t}$ is a martingale for the filtration $\cF_t$ with bounded quadratic variation, we obtain using 
 martingale exponentiation (at $s=\infty$) that

 \begin{equation}
  \bbE \left[ e^{i v(\gep)N^{(\gep)}_{[t\wedge T, T]}+\frac{v(\gep)^2}{2} \langle N^{(\gep)}\rangle_{[t\wedge T, T]}} \ | \ \cF_t \right]=1.
 \end{equation}
As a consequence we have 
\begin{equation}
   \bbE \left[ Y_t\left( e^{i v(\gep)N^{(\gep)}_{[t\wedge T, T]}+\frac{1}{2} \langle N^{(\gep)}\rangle_{[t\wedge T, T]} -\frac{1}{2}Z^A_t}- e^{-\frac{1}{2}Z^A_t}\right)  \right]=0
\end{equation}
and thus 
\begin{multline}
E_2(\gep,A)= \left| \bbE \left[ Y_t\left( e^{i v(\gep)N^{(\gep)}_{[t\wedge T, T]}} - e^{i v(\gep)N^{(\gep)}_{[t\wedge T, T]}+\frac{1}{2}\left( v(\gep)^2\langle N^{(\gep)}\rangle_{[t\wedge T, T]} -Z^A_t\right)} \right) \right] \right|\\ \le
\bbE \left[  \left| 1- e^{\frac{1}{2}\left( v(\gep)^2\langle N^{(\gep)}\rangle_{[t\wedge T, T]} -Z^A_t\right)}\right| \right].
\end{multline}
To conclude, we simply observe that both $v(\gep)^2\langle N^{(\gep)}\rangle_{[t\wedge T, T]}$
and $Z^A_t$ are bounded and both converge in probability to $Z\wedge A$ when $\gep$ goes to zero so that by dominated convergence 
\begin{equation}
\lim_{\gep\to 0}  \bbE \left[  \left| 1- e^{\frac{1}{2}\left( v(\gep)^2\langle N^{(\gep)}\rangle_{[t\wedge T, T]} -Z^A_t\right)}\right| \right] =0.
\end{equation}
This concludes the proof of \eqref{tiziro}.
\qed

 \section{Proof of Proposition \ref{quadraz2}} \label{sec:quadraz2}

 As for the proof of Proposition \ref{quadraz}, we are going to control the derivative of the brackets associated with a complex valued martingale.
 We set $$W^{(\gep)}_t:= \int_{\bbR^{d}} f(x) 
e^{\gamma X_{t,\gep}(x)- \gamma^2 K_{t,\gep}(x,y)} \ind_{\cD_{\gep}}\dd x.$$
Assuming that $\gep$ is sufficiently small the support of $f$ is included in $\cD_{\gep}$ so we
do not need to worry about the indicator function.
Using that  $(X_{t,\gep}(x))_{t\ge 0}$ is a continuous martingale, we have from Itô calculus
\begin{equation}
\dd W^{(\gep)}_t:= \gamma \int_{\bbR^d} f(x) e^{\gamma X_{t,\gep}(x)- \gamma^2 K_{t,\gep}(x,y)} \dd X_{t,\gep}(x) \dd s.
\end{equation}
Since by construction we have $\dd \langle X_{\cdot,\gep}(x),X_{\cdot,\gep}(y)\rangle_t= Q_{t,\gep}(x,y)$
where 
\begin{equation}
 Q_{t,\gep}(x,y):=\int_{\bbR^{2d}} \theta_{\gep}(x-z_1) \theta_{\gep}(y-z_2)\kappa(t|z_1-z_2|)\dd z_1 \dd z_2,
\end{equation}
this yields

\begin{equation}\label{hooks}
 \langle W^{(\gep)}, \bar W^{(\gep)}\rangle_\infty= |\gamma|^2\int^{\infty}_0 A_{t,\gep}  \dd t
 \quad \text{ and } \quad \langle W^{(\gep)}, W^{(\gep)}\rangle_\infty= \gamma^2 \int^{\infty}_0   B_{t,\gep} \dd t.
\end{equation}

 We set 
 \begin{equation}
\begin{split}
A_{t,\gep} &:= \int_{\bbR^{2d}} f(x)f(y)  Q_{t,\gep}(x,y) 
e^{\gamma X_{t,\gep}(x)+\bar \gamma X_{t,\gep}(y) +\frac{\beta^2-\alpha^2}{2}(K_{t,\gep}(x)+K_{t,\gep}(y))  }\dd x \dd y,\\
B_{t,\gep} &:=\int_{\bbR^{2d}} f(x)f(y) Q_{t,\gep}(x,y) e^{\gamma (X_{t,\gep}(x)+X_{t,\gep}(y))+\frac{\gamma^2}{2} (K_{t,\gep}(x)+ K_{t,\gep}(y))}\dd x \dd y.
\end{split}
\end{equation}
Recalling the definition of $\hat K_s$ \eqref{whops} let us define 
\begin{equation}\begin{split}
 \hat K_{s,\gep}(x,y)&:= \int_{\bbR^{2d}} \hat K(z_1,z_2) \theta_{\gep}(x-z_1)\theta_{\gep}(y-z_2)\dd z_1 \dd z_2= \int^s_0 Q_{t,\gep}(x,y)\dd t\\
  a(t,\kappa,\gamma,\gep)&:= \int_{\bbR^d} Q_{t,\gep}(z) e^{|\gamma|^2 \hat K_{t,\gep}(0,z)}\dd z.\end{split}
 \end{equation}
 We write simply $\hat K_\gep$ when $t=\infty$.
The key estimate to prove Proposition \ref{quadraz2} is the following
 \begin{proposition}\label{lealebprim}
 
We have the following convergences in $L_1$
 \begin{equation}\begin{split}\label{grandz2}
  \limtwo{t\to \infty}{\gep \to 0} a(t,\kappa,\gamma,\gep)^{-1} A_{t,\gep}&= M^{(2\alpha)}_0(e^{K_0}f^2),\\
    \limtwo{t\to \infty}{\gep \to 0} a(t,\kappa,\gamma,\gep)^{-1} B_{t,\gep}&= 0.
 \end{split}\end{equation}

 \begin{proof}[Proof of Proposition \ref{quadraz2}]
 
 Let us prove first that we 
 have 
 \begin{equation}\label{xlinkbis}
  \lim_{\gep \to 0} |\gamma|^2 v(\gep,\theta)^ 2\int^{\infty}_0   a(t,\kappa,\gamma,\gep) \dd t= 2 e^{-|\gamma|^2 j_{\kappa}}
 \end{equation}
When $|\gamma|^2>d$ we have
 \begin{multline}\label{lorrib}
 |\gamma|^2\int^{\infty}_{0} a(t,\kappa,\gamma,\gep)= \int_{\bbR^d} \left(e^{|\gamma|^2 \hat K_{\gep}(0,z)}-1\right) \dd z\\=\gep^{d-|\gamma|^2} \int_{|y|\le \gep^{-1}+2} \left( e^{|\gamma|^2 \left( \hat K_{\gep}(0, \gep y)- \log (1/\gep) \right)}- \gep^{|\gamma|^2} \right)\dd y.
 \end{multline}
 Now integrating $\gep^{|\gamma|^2}$ yield something tending to zero so we just need to worry about the first term in the integral.
To compute the other term, we consider the following decomposition $$\hat K(x,y)= \log \left(\frac{1}{|x-y|\vee 2}\right) + \left[ \int^{\infty}_0 \kappa(e^t|x-y|) \dd t- \log \left(\frac{1}{|x-y|\vee 2}\right)          \right].$$ 
Let us call $U(x,y)$ the second term, it is uniformly bounded and equal to $-j_{\kappa}$ on the diagonal (recall \eqref{jkappa}). When $|y|\le 2+\gep^{-1}$ and $\gep\le 1/4$ we have
 \begin{equation}\label{convoluz}
\hat K_{\gep}(0, \gep y)= \int_{\bbR^{2d}} \theta_{\gep}(z_1)\theta_{\gep}(z_2)
\left(\log \frac{1}{|z_1+z_2-\gep y|}+ U(    \gep y-z_1+z_2)  \right) \dd z_1 \dd z_2 
\end{equation}
Performing a change of variable and recalling \eqref{ltheta} we have
\begin{equation}
 \int_{\bbR^{2d}} \theta_{\gep}(z_1)\theta_{\gep}(z_2)
\log \frac{1}{|z_1+z_2-\gep y|} \dd z_1 \dd z_2  = \ell_{\theta}(y)+\log (1/\gep),
\end{equation}
while
\begin{equation}
\lim_{\gep \to 0}  \int_{\bbR^{2d}} \theta_{\gep}(z_1)\theta_{\gep}(z_2) U(    \gep y-z_1+z_2)= -j_{\kappa}.
\end{equation}
Has a conclusion we obtain that 
\begin{equation}\begin{split}
\lim_{\gep \to 0}\left( \hat K_{\gep}(0, \gep y)- \log (1/\gep) \right)&=\ell_{\theta}(y)\\
\left|\left( \hat K_{\gep}(0, \gep y)- \log (1/\gep) \right) - \ell_{\theta}(y)\right|&\le C.
\end{split}\end{equation}
By dominated convergence, this proves that 
\begin{equation}
 \lim_{\gep\to 0}\int_{\bbR^d} e^{|\gamma|^2 \left( \hat K_{\gep}(0, \gep y)- \log (1/\gep) \right)} \ind_{\{|y|\le 2\gep^{-1}+2\}} \dd y= e^{-|\gamma|^2 j_{\kappa}}\int_{\bbR^d}e^{-|\gamma|^2 \ell_{\theta}(y)} \dd y.
\end{equation}
and hence recalling \eqref{lorrib} that \eqref{xlinkbis} holds.
Now when $|\gamma|=\sqrt{d}$ we split the integral 
$\int_{\bbR^d} \left(e^{d \hat K_{\gep}(0,z)}-1\right)\ind_{\{|z|\le 1+2\gep\}} \dd z$ in three parts  (again we can throw the $-1$ out since it does not yield a significant contribution) controled in \eqref{1cut}, \eqref{2cut} and \eqref{3cut} respectively.
Since $\hat K_{\gep}(0,z)\le \log (1/\gep)+C$ for all $|z|$ we have 
\begin{equation}\label{1cut}
 \int_{|z|\le \gep \log \log (1/\gep)}e^{d \hat K_{\gep}(0,z)}\dd z \le  C \log \log (1/\gep)^d. 
\end{equation}
Using \eqref{convoluz} (replacing $\gep y$ by $z$) we see that 
$\hat K_{\gep}(0,z)\le \log^{(3)} (1/\gep)+C$ when $|z|\ge \log \log (1/\gep)$.
This yields 
\begin{equation}\label{2cut}
  \int_{\log \log (1/\gep)\le |z|\le 1+2\gep}e^{d \hat K_{\gep}(0,z)}\dd z\le C \log \log (1/\gep)^d.
\end{equation}
Finally, using \eqref{convoluz} ($U$ is a Lipshitz function) we obtain that 
\begin{equation*}
 \left\{ |z|\in (\gep \log\log (1/\gep),\log \log(1/\gep)^{-1}) \right\} \ \Rightarrow \   \left\{ |\hat K_{\gep}(0,z)- \log|z|+j_{\kappa}|\le C \log \log (1/\gep)^{-1} \right\}
\end{equation*}
so that 
\begin{equation}\label{3cut}
 \lim_{\gep\to 0}\frac{ e^{d j_{\kappa}} \int_{|z|\in (\gep \log\log (1/\gep),\log \log(1/\gep)^{-1})}e^{d \hat K_{\gep}(0,z)}\dd z}{ \int_{|z|\in (\gep \log\log (1/\gep),\log \log(1/\gep)^{-1})}  |z|^{-d}\dd z}=1
\end{equation}
The denominator is asymptotically equivalent to  $ \frac{2\pi^{d/2}}{\gG(d/2)}\log (1/\gep)$
(the prefactor is simply the volume of the $d-1$ dimensional sphere of radius one).
Combined with \eqref{1cut} and \eqref{2cut}, this yields \eqref{xlinkbis} also in that case.

\medskip

Now recalling that $N^{(\gep)}= \mathfrak{Re} W^{(\gep)}e^{-i \go}$
we have
  \begin{equation}
  \langle N^{(\gep)}\rangle_\infty= \frac{1}{2}\langle W^{(\gep)}, \bar W^{(\gep)} \rangle_\infty + \frac{1}{2} \mathfrak{Re}(e^{-2i\go}\langle W^{(\gep)},  W^{(\gep)} \rangle_\infty)
\end{equation}
In view of \eqref{xlinkbis}, it is sufficient to prove that 
\begin{equation}\label{lelimitz}
 \lim_{\gep \to 0} \frac{\langle W^{(\gep)}, \bar W^{(\gep)} \rangle_\infty}{|\gamma|^2 \int^{\infty}_0   a(t,\kappa,\gamma,\gep) \dd t}=  M^{(2\alpha)}_0(e^{K_0}f^2) \text{ and }  \lim_{\gep \to 0}  \frac{\langle W^{(\gep)}, W^{(\gep)} \rangle_\infty}{\int^{\infty}_0   a(t,\kappa,\gamma,\gep) \dd t}=0.
\end{equation}
Let us set $t_\gep:= \sqrt{\log (1/\gep)}$ (we need $1 \ll t_{\gep}\ll \log (1/\gep)$).
From Proposition \ref{lealebprim} and \eqref{hooks} we have
\begin{equation}
  \lim_{\gep \to 0}\frac{\langle W^{(\gep)}, \bar W^{(\gep)} \rangle_{[t_{\gep},\infty)}}{|\gamma|^2 \int^{\infty}_{t_{\gep}}   a(s,\kappa,\gamma,\gep) \dd s}  \quad  \text{ and } \quad  
   \frac{\langle W^{(\gep)}, W^{(\gep)} \rangle_{[t_{\gep},\infty)}}{\int^{\infty}_{t_{\gep}}   a(s,\kappa,\gamma,\gep) \dd s}=0.
\end{equation}
To conclude we just need to show that $\bbE \left[\langle W^{(\gep)}, \bar W^{(\gep)} \rangle_{t_{\gep}} \right]$ and $\int_{0}^{t_{\gep}}   a(s,\kappa,\gamma,\gep) \dd s$
are of a smaller order than $\int^{\infty}_{0}   a(s,\kappa,\gamma,\gep) \dd s$.
This is done like in the proof of Proposition \ref{quadraz} in Section \ref{sec:quadraz} (the reader can check that these terms are of order 
 order $e^{(|\gamma|^2-d)t_{\gep}}$ (if $|\gamma|>\sqrt{d}$) or $t_{\gep}$ (if $|\gamma|=\sqrt{d}$), details are left to the reader.

\end{proof}

\end{proposition}

\subsection{Proof of Proposition \ref{lealebprim}}
 We proceed as in the case without convolution (Proposition \eqref{lealeb}) in several steps.
 The case of $A_{t,\gep}$ is the more delicate so we treat it in full details and then sketch briefly the proof for $B_{t,\gep}$.
 We assume that $K_0\equiv 0$ for the sake of simplifying notation.
 For the need of the computation we need to fix $u<t$ such that
 $X_{u,\gep}(x)-X_{u,\gep}(y)$ is small whenever $Q_{t,\gep}(x,y)>0$.
 Keeping \eqref{estimao} and Lemma \ref{localike} in mind, 
 a convenient choice is 
 \begin{equation}
 u:= u(t,\gep)= \left(t\wedge \log  \frac{1}{\gep}\right) - \log \left( t\wedge \log \frac{1}{\gep}\right).
 \end{equation}
We set  $$k(t,\gep):=K_{t,\gep}(x,x)$$ (since $K_0\equiv 0$, the field is translation invariant so it does not depend on $x$).
Note that from \eqref{localike} we have
\begin{equation}
\left| k(t,\gep)- \left(t\wedge \log  \frac{1}{\gep}\right)\right|\le C,
\end{equation}
so that we have in particular in the $\gep\to 0$ and $t\to \infty$ regime    
\begin{equation}\label{interscale}
1\ll k(t,\gep)-k(u,\gep) \ll k(t,\gep)   \quad  \text{ and } \quad |k(u,\gep)-u|\le C
\end{equation}
We define, using \eqref{intervalnotation}
 \begin{equation}\begin{split}
 A^{(1)}_{t,\gep} &:= \int_{\bbR^{2d}} f(x)^2  Q_{t,\gep}(x,y) 
e^{\gamma X_{t,\gep}(x)+\bar \gamma X_{t,\gep}(y) +(\beta^2-\alpha^2)k(t,\gep) }\dd x \dd y,\\
  A^{(2)}_{t,\gep} &:= \int_{\bbR^{2d}} f(x)^2  Q_{t,\gep}(x,y) 
e^{\gamma X_{u,\gep}(x)+\bar \gamma X_{u,\gep}(y) +(\beta^2-\alpha^2)k(u,\gep)+|\gamma|^2 K_{[u,t],\gep}(x,y) }\dd x \dd y,\\
&= \bbE\left[A^{(1)}_{t,\gep} \ | \ \cF_u \right]
 \end{split}
 \end{equation}
Let us also introduce 
\begin{equation}
 M^{(2\alpha)}_{t,\gep}(f^2):=   \int_{\bbR^{2d}} f(x)^2  
 e^{2\alpha X_{u,\gep}(x)- 2\alpha^2 k(t,\gep)}\dd x \dd y
\end{equation}
Writing $a(t,\gep)$ for $a(t,\kappa,\gamma,\gep)$ we have 
 \begin{multline}\label{zyup}
| A_{t,\gep}-a(t,\gep)M^{(2\alpha)}_0(f^2)|
  \le | A_{t,\gep}-A^{(1)}_{t,\gep}| + | A^{(1)}_{t,\gep}-A^{(2)}_{t,\gep}| \\+| A^{(2)}_{t,\gep}-a(t,\kappa,\gamma,\gep) M^{(2\alpha)}_{t,\gep}| 
  +a(t,\gep) | M^{(2\alpha)}_{t,\gep}(f^2)-  M^{(2\alpha)}_0(f^2)|
 \end{multline}
and we need to show that all of the four terms in the r.h.s.\ rescaled by $a(t,\kappa,\gamma,\gep)$ tend to zero. This is done in four separate steps.  The argument for the first three terms are essentially the same as for the proof of Proposition \ref{lealeb} in Section \ref{sec:quadraz}.

\subsubsection*{Step 1: Bounding $\bbE \left[| A_{t,\gep}-A^{(1)}_{t,\gep}| \right]$}

Let us prove first that 
\begin{equation}\label{termo1}
 \limtwo{\gep \to 0}{t\to \infty} a(t,\gep)^{-1} \bbE[  |A^{(1)}_{t,\gep}-A_{t,\gep}|]=0
\end{equation}
Now since  $Q_{t,\gep}(x,y)=0$ when $|x-y|> e^{-t}+2\gep$, letting $w_f$ denote the modulus of continuity of $f$ we have
\begin{multline}
 \bbE\left[ |A_t-A^{(1)}_t| \right] \le w_f(e^{-t}+2\gep)\int_{\bbR^{2d}} |f(x)|    Q_{t,\gep}(x,y) e^{|\gamma|^2 K_{t,\gep}(x,y)}   \dd x \dd y \\
 = w_f(e^{-t}+2\gep) \|f \|_1 a(t,\gep).
\end{multline}
 We conclude the proof by simply observing that  $w_f(e^{-t}+2\gep)$ tends to zero.

\subsubsection*{Step 2: Bounding $\bbE \left[| A^{(1)}_{t,\gep}-A^{(2)}_{t,\gep}| \right]$}

Let us now prove that 
\begin{equation}\label{tremi}
  \limtwo{\gep \to 0}{t\to \infty} a(t,\gep)^{-1} \bbE\left[ A^{(1)}_{t,\gep}-A^{(2)}_{t,\gep}\right]=0
\end{equation}
We provide details only on the more delicate case $\alpha\in[\sqrt{d}/2,\sqrt{d/2})$ 
(the reader can repeat the procedure in Section \ref{sec:quadraz} for the case  $\alpha\in [0, \sqrt{d}/2)$).
Let us set  
\begin{equation}\begin{split}\label{zetaxitgep}
\xi_{t,\gep}(x,y)&:=  e^{\gamma X_{t,\gep}(x)+\bar \gamma X_{t,\gep}(y) +(\beta^2-\alpha^2)k(t,\gep)} - \bbE[e^{\gamma X_{t,\gep}(x)+\bar \gamma X_{t,\gep}(y) +(\beta^2-\alpha^2)k(t,\gep)} \ | \ \cF_u],\\ 
\zeta_{t,\gep}(x,y)&:=Q_{t,\gep}(x,y)f(x)^2\xi_t(x,y).
\end{split}\end{equation}
We fix $\gl$ satisfying \eqref{lambdabal} and set similarly to \eqref{defcat}
\begin{equation}
\cA_{t,\gep}(x)=\{ X_{u,\gep}(x)\le \gl k(u,\gep) \}
\end{equation}
(recall that now $u$ depends also on $\gep$).
We have
\begin{equation}\label{ytyt}
 A^{(1)}_{t,\gep}-A^{(2)}_{t,\gep}=
 \int_{\bbR^{2d}} \zeta_{t,\gep}(x,y)\ind_{\cA^{\cc}_{t,\gep}(x)} \dd x \dd y
 + \int_{\bbR^{2d}} \zeta_{t,\gep}(x,y)\ind_{\cA_{t,\gep}(x)} \dd x \dd y=: W^{(1)}_{t,\gep}+W^{(2)}_{t,\gep} 
\end{equation}
To prove \eqref{tremi} we are going to bound the $L_1$ norm of the first term and the $L_2$ norm of the second term.
We have using the Cameron-Martin formula (Proposition \ref{cameronmartinpro})
\begin{multline}
 \bbE \left[ |\xi_{t,\gep}(x,y)| \ind_{\cA^{\cc}_{t,\gep}(x)}\right]
 \le 2 \bbE\left[ e^{\alpha \left(X_{t,\gep}(x)+ X_{t,\gep}(y)\right)-(\beta^2-\alpha^2)k(t,\gep)}\ind_{\cA^{\cc}_{t,\gep}(x)} \right]\\
 =  2e^{\beta^2 k(t,\gep) +\alpha^2 K_{t,\gep}(x,y)} \bbP \left[ X_{u,\gep}
 > (\gl - \alpha)k(u,\gep)-\alpha K_{u,\gep}(x,y)\right]\\
 \le  2e^{|\gamma|^2 k(t,\gep)} \bbP \left[ X_{u,\gep}(x)
 > (\gl - 2\alpha)k(u,\gep)\right].
\end{multline}
Using the Gaussian tail bound \eqref{gtail} , this entails that 
\begin{multline}
\bbE[|W^{(1)}_{t,\gep}|]\le
\bbE\left[\int_{\bbR^{2d}} \left|  \zeta_{t,\gep}(x,y) \right|\ind_{\cA^{\cc}_{t,\gep}(x)} \dd x \dd y\right]\\ \le 2e^{|\gamma|^2 k(t,\gep)-\frac{(\gl - 2\alpha)^2k(u,\gep)}{2}}
\int_{\bbR^{2d}} f(x)^2 Q_{t,\gep}(x,y) \dd x \dd y
 = a(t,\gep)e^{-\frac{(\gl - 2\alpha)^2k(u,\gep)}{2}} \| f\|^2_2.
\end{multline}
For the second moment computation  $W^{(2)}_{t,\gep}$, note that the range of correlation of $X_{[u,t],\gep}$ is smaller than $e^{-u}+2\gep$ and that $Q_{t,\gep}(x,y)=0$ when $|x-y|\ge e^{-t}+2\gep$ so that 

\begin{equation}
|x_1-x_2|\ge 6 \gep+2e^{-t}+e^{-u} \quad \Rightarrow \quad  \bbE[ \zeta_{t,\gep}(x_1,y_1)\bar \zeta_{t,\gep}(x_2,y_2) \ | \ \cF_u]=0.
\end{equation}
Hence taking $\gep$ sufficiently small and $t$ sufficiently large, we have, whenever $|x_1-x_2|\le 2e^{-u}$ 
\begin{equation}
  \bbE[ \zeta_{t,\gep}(x_1,y_1)\bar \zeta_{t,\gep}(x_2,y_2) \ind_{\cA_{t,\gep}(x_1) \cap \cA_{t,\gep}(x_2)}]=0.
\end{equation}
When $x_1$ and $x_2$ are closer to each other, we bound the covariance by the  mean of the variances. We have
\begin{multline}
   \bbE[ |\xi_{t,\gep}(x_,y)|^2  \ind_{\cA_{t,\gep}(x)} ]\le 
   \bbE[ e^{2\alpha (X_{t,\gep}(x)+X_{t,\gep}(y))-2(\beta^2-\alpha^2)k(t,\gep)} \ind_{\cA_{t,\gep}(x)} ] 
  \\ =e^{\left(2 |\gamma|^2 k(t,\gep)+4\alpha^2 K_{t,\gep}(x,y)\right)} \bbP[ X_{u,\gep}\le (\gl-2\alpha)k(u,\gep)-2\alpha K_{u,\gep}(x,y) ]
\end{multline}
Our definition of $u$ and Lemma \ref{localike} implies that  
there exists some positive constant $C$ satisfying 
$$  Q_{t,\gep}(x,y)>0   \quad \Rightarrow \quad K_{u,\gep}(x,y) \ge  k(u,\gep)-C $$ Thus using Gaussian tail bounds
\eqref{gtail} we have, whenever  $Q_{t,\gep}(x,y)>0$ 
\begin{equation}
  \bbP[ X_{u,\gep}\le (\gl-2\alpha)k(u,\gep)-2\alpha K_{u,\gep}(x,y) ]\le Ce^{-\frac{(\gl-4\alpha)^2k(u,\gep)}{2}}.
\end{equation}
Thus  we obtain that
\begin{multline}
 \bbE \left[ |W^{(2)}_{t,\gep}|^2 \right]\le C  e^{(2 |\gamma|^2+4\alpha^2) k(t,\gep)- \frac{(\gl-4\alpha)^2k(u,\gep)}{2}}\\ \times
 \int_{\bbR^{4d}} f(x_1)^2f(x_2)^2  Q_{t,\gep}(x_1,y_1)Q_{t,\gep}(x_2,y_2)  \ind_{\{ |x_1-x_2|\le 2e^{-u}\}} \dd x_1 \dd x_2 \dd y_1 \dd y_2.
\end{multline}
Now the integral above can be bounded by
\begin{equation}
  C \|f \|^4_4 e^{-d u} \left(\int_{\bbR^d} Q_{t,\gep}(0,z) \dd z\right)^2.
\end{equation}
Thus we obtain that 
\begin{equation}
  \bbE \left[ |W^{(2)}_{t,\gep}|^2 \right]\le C  \|f \|^4_4 a(t,\gep)^2  e^{4\alpha^2 k(t,\gep)- du-\frac{(\gl-4\alpha)^2k(u,\gep)}{2}}.
\end{equation}
It only remains to prove that the exponential terms goes to zero. This is simply a consequence of  \eqref{lambdabal} and \eqref{interscale} (which essentially allows to replace $u$ and $k(u,\gep)$ by $k(t,\gep)$).

\subsubsection*{Step 3: Bounding $\bbE \left[|A^{(2)}_{t,\gep}-M^{(2\alpha)}_{u,\gep}| \right]$}
We prove now that
\begin{equation}\label{tremax}
  \limtwo{\gep \to 0}{t\to \infty} a(t,\gep)^{-1} \bbE\left[ A^{(1)}_{t,\gep}-A^{(2)}_{t,\gep}\right]=0.
\end{equation}
This much easier, we have from Jensen's inequality
\begin{multline}
 \bbE \left[|A^{(2)}_{t,\gep}-a(t,\gep)M^{(2\alpha)}_{u,\gep}| \right]
 \le \int_{\bbR^2} f(x)^2 Q_t(x,y)e^{|\gamma|^2 K_{t,\gep}(x,y)} \\
 \times \bbE\left[ e^{2\alpha X_{u,\gep}(x)-2\alpha^2k(u,\gep)}|e^{ \bar \gamma (X_{u,\gep}(y)- X_{u,\gep}(x))+|\gamma|^2(k(u,\gep)-K_{u,\gep}(x,y))} -1|   \right]\dd x \dd y
\end{multline}
Using the Cameron-Martin formula (Proposition \ref{cameronmartinpro}) we have 
\begin{multline}
 \bbE\left[ e^{2\alpha X_{u,\gep}(x)-2\alpha^2k(u,\gep)}|e^{ \bar \gamma (X_{u,\gep}(y)- X_{u,\gep}(x))+|\gamma|^2(k(u,\gep)-K_{u,\gep}(x,y))} -1|   \right]
 \\=   \bbE\left[ |e^{ \bar \gamma (X_{u,\gep}(y)- X_{u,\gep}(x))-\bar \gamma^2(k(u,\gep)-K_{u,\gep}(x,y))} -1|   \right]
\end{multline}
and finally 
\begin{multline}
  \bbE\left[ |e^{ \bar \gamma (X_{u,\gep}(y)- X_{u,\gep}(x))-\bar \gamma^2(k(u,\gep)-K_{u,\gep}(x,y))} -1|   \right]^2 \\ \le   \bbE\left[ |e^{ \bar \gamma (X_{u,\gep}(y)- X_{u,\gep}(x))-\bar \gamma^2(k(u,\gep)-K_{u,\gep}(x,y))} -1|^2   \right]= e^{2|\gamma|^2 (k(u,\gep)-K_{u,\gep}(x,y)) }-1. 
\end{multline}
Now, using Lemma \ref{localike}, and our definition of $u$, we have whenever $|x-y|\le e^{t}+2\gep$
\begin{equation}
 k(u,\gep)-K_{u,\gep}(x,y)\le Ce^{-u}|x-y|+ \eta(|x-y|)\le  \delta(t,\gep)
\end{equation}
where $\delta(t,\gep)$ when $t$ goes to infinity and $\gep$ does to $0$. Altogether we obtain that 
\begin{equation}
  \bbE \left[|A^{(2)}_{t,\gep}-a(t,\gep)M^{(2\alpha)}_{u,\gep}| \right] \le C \|f\|^2_2
  a(t,\gep) \sqrt{\delta(t,\gep)},
\end{equation}
concluding the proof of \eqref{tremax}.

\subsubsection*{Step 4: Bounding  $\bbE\left[| M^{(2\alpha)}_{t,\gep}-  M^{(2\alpha)}_0| \right]$}

Let us finally  discuss the fourth term. We want to prove the following
\begin{equation}\label{finaltruc}
 \limtwo{\gep \to 0}{t\to \infty} \bbE\left[| M^{(2\alpha)}_{t,\gep}(f^2)-  M^{(2\alpha)}_0(f^2)| \right]=0.
\end{equation}
Omitting $f^2$ for readability we have
\begin{equation}
 \bbE\left[| M^{(2\alpha)}_{t,\gep}-  M^{(2\alpha)}_0| \right]= \bbE\left[| M^{(2\alpha)}_{t,\gep}- M^{(2\alpha)}_t| \right]+ \bbE \left[ |M^{(2\alpha)}_t-M^{(2\alpha)}_0| \right].
\end{equation}
The second term tends to zero when $t\to \infty$ thanks to 
Theorem \ref{zabbid2}.
As for the first one, note that we have  $$M^{(2\alpha)}_{t,\gep}- M^{(2\alpha)}_t= \bbE \left[ M^{(2\alpha)}_{\gep}- M^{(2\alpha)}_0  \ | \ \cF_t\right]$$ so that by Jensen inequality for conditional expectation we have
\begin{equation}
 \bbE\left[| M^{(2\alpha)}_{t,\gep}- M^{(2\alpha)}_t| \right]\le 
 \bbE[  |M^{(2\alpha)}_{\gep}- M^{(2\alpha)}_0 |]
\end{equation}
and by Theorem \ref{zabbid}, the right-hand side tends to zero when $\gep$ goes to zero, yielding 
\eqref{finaltruc}.

\subsubsection*{Final step: Bounding $B_{t,\gep}$}

Finally to bound $B_{t,\gep}$ we proceed as we did for $B_t$.
We show that  
\begin{equation}\begin{split}
\limtwo{\gep\to 0}{t\to0}a(t,\gep)^{-1} &\bbE[| \bbE [B_{t,\gep} \ | \ \cF_u]|]=0, \\
\limtwo{\gep\to 0}{t\to0}
a(t,\gep)^{-1} & \bbE[|B_{t,\gep} -\bbE [B_{t,\gep} \ | \ \cF_u]|]=0.
\end{split}
\end{equation}
For the first line we repeat the argument of \textit{Step 3} above while for the second line
we use the same proof as in \textit{Step 2}.

\qed

 \noindent {\bf Acknowledgement:} The author is grateful to Paul Gassiat for indicating the reference \cite{jacodsh} for the proof of Theorem \ref{TLC} and letting him know about the notion of stable convergence which is the adequate framework to present the main result of this paper. He thanks J.F.\ Le Gall, Rémi Rhodes and Vincent Vargas for enlightening comments. This  work was  realized  during  H.L.\ extended stay in Aix-Marseille University funded by the European Union’s Horizon 2020 research and innovation program under the Marie Sklodowska-Curie grant agreement No 837793. 

\appendix

\section{
Proof of Lemma \ref{lesasump}} \label{approxapp}

Let $\cD'$ and $\delta$ be fixed.
We consider $\kappa_0$ a $C^{\infty}$ kernel satisfying the assumptions listed below \eqref{ilaunestar}. Setting $\eta:=(s-d)/2$ we consider $K^{\delta}$ defined by
\begin{equation}
 K^{\delta}(x,y):= K(x,y)+ \eta\delta  \int^{\infty}_0 e^{-\eta t} \kappa_0(e^t|x-y|)\dd t.
\end{equation}
Note that $K'$  satisfies assumptions (A)-(B) since by construction
\begin{equation}
 (K^{\delta}-K)(x,y)=  \eta\delta  \int^{\infty}_0 e^{-\eta t} \kappa_0(e^t|x-y|)\dd t
\end{equation}
is a positive definite kernel and is equal to $\delta$ on the diagonal.
To prove that $K^{\delta}$ can be written in the form \eqref{ilaunestar}
we need to prove the following claim

\begin{lemma}\label{lelezero}
If $t_0$ is sufficiently large then
$$ K^{\delta}(x,y)-\int^{\infty}_{t_0} \kappa_0(e^t|x-y|) \dd t=K^{\delta}(x,y)-\int^{\infty}_{0} \kappa_0(e^{t_0+u}|x-y|) \dd u.$$
is positive definite and H\"older continuous.
\end{lemma}

\noindent The lemma implies that \eqref{ilaunestar} is satisfied for $K^{\delta}(x,y)$ with 
$\kappa(u)=\kappa_0(e^{t_0}u)$.
It is a consequence of the two following estimates.

\begin{lemma}\label{lepremlem}
Given $s>d$, there exists a constant $C(s,\kappa_0,K)$ such that
 for any $\varphi\in C^{\infty}(\cD')$   we have
 
 \begin{equation}\label{notooneg}
  \int_{\bbR^{2d}} \left( K(x,y)-\int^{\infty}_{t_0} \kappa_0(e^t|x-y|) \dd t\right)
  \varphi(x)\varphi(y) \dd x \dd y \ge - C e^{-\frac{(d-s)}{2}t_0}  \| \varphi \|^2_{H^{-s/2}(\bbR^d)}
 \end{equation}

\end{lemma}

\begin{lemma}\label{leseclem}
For any $\varphi\in C^{\infty}_c( \cD')$  we have
 \begin{equation}\label{tipop}
  \int \left( \eta\delta  \int^{\infty}_0 e^{-\eta t} \kappa_0(e^t|x-y|)\dd t\right)
  \varphi(x)\varphi(y) \dd x \dd y \ge \delta  c(\kappa_0,s)  \| \varphi \|^2_{H^{-s/2}(\bbR^d)}
 \end{equation}
 
\end{lemma}

\noindent Let us first deduce Lemma \ref{lelezero} from Lemma \ref{lepremlem} and Lemma \ref{leseclem}
\begin{proof}[Proof of Lemma \ref{lelezero}]
 First note that 
 \begin{multline}
  K^{\delta}(x,y)-\int^{\infty}_{t_0} \kappa_0(e^t|x-y|) \dd t
  \\
  =  L(x,y)+ \eta\delta  \int^{\infty}_0 e^{-\eta t} \kappa_0(e^t|x-y|)\dd t+
 \left( \log \frac{1}{|x-y|} -\int^{\infty}_{t_0} \kappa_0(e^t|x-y|) \dd t\right).
  \end{multline}
Each of the three terms are H\"older continuous on $\cD'\times \cD'$ ($L(x,y)$ because of its Sobolev regularity combined with Morrey's inequality, the other two can be checked by hand).
Now combining \eqref{notooneg} and \eqref{tipop} we have
\begin{multline}
   \int_{\bbR^{2d}} \left( K^{\delta}(x,y)-\int^{\infty}_{t_0} \kappa_0(e^t|x-y|) \dd t\right)
  \varphi(x)\varphi(y) \dd x \dd y \\\ge \left(\delta c(\kappa_0,s)- C e^{-\frac{(d-s)}{2}t_0}\right) \| \varphi \|^2_{H^{-s/2}(\bbR^d)}
\end{multline}
and the r.h.s.\ is positive if $t_0$ is chosen to be sufficiently large.

\end{proof}

\begin{proof}[Proof of Lemma \ref{lepremlem}]
First we notice that we have 
\begin{multline}
 K(x,y)=  \left(L(x,y)+\log \frac{1}{|x-y|}- \int^{\infty}_{0} \kappa(e^t|x-y|)\dd t \right)+\int^{\infty}_{0} \kappa(e^t|x-y|)\dd t\\=:  \tilde L(x,y)+ \tilde K(x,y),
\end{multline}
where $\tilde L\in H^{s}_{\mathrm{loc}}(\cD'\times \cD')$ (this is simply because 
$ \log \left( \frac{1}{|x-y|}\right)- \int^{\infty}_0 \kappa(e^t|x-y|)\dd t$ is a $C^{\infty}$ function).
Now we consider $\gep>0$ sufficiently small so that $\cD' \subset \cD_{\gep}$ \eqref{cdgep}. Recalling \eqref{labig}, we are going to use the $\gep$ subscript notation for convolution with $\theta_{\gep}$ on both coordinates for $\tilde L$ also.
Since $K_{\gep}$ is definite positive, it is sufficient to show that the inequality \eqref{notooneg} holds for the following kernel
\begin{multline}
  K(x,y)-K_{\gep}(x,y)-\int^{\infty}_{t_0} \kappa_0(e^t|x-y|) \dd t
  \\  =(\tilde L-\tilde L_{\gep})(x,y)
 +\int_{0}^{t_0} \kappa_0(e^t|x-y|) \dd t- \tilde K_\gep(x,y) 
  \end{multline}
Now using \cite[Lemma 4.6]{junnila2019}, we have for any $\varphi \in C^{\infty}_c(\cD')$
\begin{equation}
\left| \int \varphi(x) \varphi(y) (\tilde L-\tilde L_{\gep})(x,y) \dd x \dd y\right|
 \le C_{\cD'} \|\tilde L-\tilde L_{\gep} \|_{H^{s}(\bbR^{2d})} \|\varphi \|^2_{H^{-s/2}(\bbR^{2d})}.
\end{equation}
To control the remaining part,
let us set 
\begin{equation}
 \psi(x):= \int \left(\tilde K_\gep(x,y) - \int_{0}^{t_0} \kappa_0(e^t|x-y|) \dd t\right) \varphi(y) \dd y.
\end{equation}
We have using Plancherel Theorem
\begin{multline}\label{four}
 \int_{\bbR^{2d}} \varphi(x) \varphi(y)\left(\int_{0}^{t_0} \kappa_0(e^t|x-y|) \dd t- \tilde K_\gep(x,y)\right) \dd x \dd y
 \\=-\int_{\bbR^d} \varphi(x)\psi(x) \dd x = -(2\pi)^{-d}\int_{\bbR^d}  \hat \psi(\xi)\hat \phi(\xi)\dd \xi.
\end{multline}
Now using the formula for Fourier transform of convolution and rescaled functions,we have
\begin{equation}\label{fourier}
 \hat \psi(\xi)=  \left( \int^{\infty}_0 \left(|\hat \theta(\gep \xi)|^2-\ind_{\{t\le t_0\}}\right)     e^{-dt}\hat \kappa_0(e^{-t} \xi)\dd t \right)\hat \varphi(\xi)=:T(\xi) \hat \varphi(\xi)
\end{equation}
To conclude we only need an upper bound on $T(\xi)$. Note that since $|\hat \theta(\xi)|\le 1$ for every $\xi$ and, since by Bochner's Theorem $\hat \kappa_0$ is pointwise real and non-negative, we have
\begin{equation}
  T(\xi)\le |\hat \theta(\gep \xi)|^2  \int^{\infty}_{t_0}  e^{-dt}\hat \kappa_0(e^{-t} \xi)\dd t  \le \frac{1}{d} \hat \theta(\gep \xi)^2   e^{-dt_0}.
\end{equation}
We can now set $\gep=e^{-\frac{(s+d)t_0}{2s}}$.
Now since $\theta$ is $C^{\infty}$, $\hat \theta$ decays faster than any polynomial, and hence
we can find a constant $C$ (depending on $\theta$ and $s$) which is such that 
\begin{equation}
T(\xi)\le C e^{-d t_0}(1+ |\xi|^2  \gep^2)^{-s/2}
  \le   C e^{-(d-s)t_0/2}(1+\xi^2)^{-s},
\end{equation}
which is sufficient to conclude.

\end{proof}

\begin{proof}[Proof of Lemma \ref{leseclem}]
 Following the same computation as in \eqref{four}-\eqref{fourier}
 we obtain that 
 \begin{multline}
    \int \left( \eta\delta  \int^{\infty}_0 e^{-\eta t} \kappa_0(e^t|x-y|)\dd t\right)
  \varphi(x)\varphi(y) \dd x \dd y \\ =
  \eta\delta \int \left( \int^{\infty}_0  e^{-(\eta+d) t} \hat \kappa_0(e^{-t}\xi) \dd t \right) \hat \phi(\xi)^2 \dd \xi.
 \end{multline}
Now as $\hat \kappa_0$ is non-negative (and positive around $0$) we have 
\begin{equation}
  \eta\int^{\infty}_0  e^{-(\eta+d) t} \hat \kappa_0(e^{-t}\xi) \dd t \ge c_\eta \left( 1+ |\xi|^2\right)^{-(d+\eta)/2},
\end{equation}
which is sufficient to conclude.

\end{proof}

\section{Proof of Proposition \ref{tiggh}} \label{tightness}
In order to check the tightness we work  with the Fourier transform $\hat M^{(\gamma,\rho)}_{\gep}(\xi)$ of $M^{(\gamma)}_{\gep}(\rho \cdot)$ which is almost surely finite. Most of the time we will omit the dependence in $\gamma,\rho$ for better readability, and simply write $v(\gep)$ for $v(\gep,\theta,\gamma)$.

\medskip

We need to show that $(v(\gep)M^{(\gamma)}_{\gep}(\rho \  \cdot))$ is tight 
in $H^{-u}(\bbR^d)$, which, by isometry, is equivalent to showing that $v(\gep)\hat M_{\gep}$ is tight in the space $\hat H^{-u}(\bbR^d):=L^2(\bbR^d, (1+ |\xi|^2)^{-u} \dd \xi )$. To prove the later statement we are going to use the following variant of the Frechet-Kolmogorov compactness criterion (see e.g.\ \cite[Theorem 4.26]{Brezis}).

\begin{proposition}\label{criterion}
A  subset of $K$ of 
 $\hat H^{-s}(\bbR^d)$ is relatively compact if and only if it satisfies the two following 
 conditions
\begin{itemize}
 \item [(i)] $\lim_{R\to \infty}\sup_{\varphi\in K}\int_{|\xi|>R} |\varphi(\xi)|^2(1+ |\xi|^2)^{-u}\dd \xi=0.$
 
 \item [(ii)] $\lim_{a\to 0}\sup_{\varphi\in K}\int_{\bbR^d} |\varphi(\xi+a)- \varphi(\xi)|^2(1+ |\xi|^2)^{-u}\dd \xi=0.$
\end{itemize}

\end{proposition}

\noindent Now the tightness of $\hat M_{\gep}$ will be proved using the following simple estimates.

\begin{lemma}\label{kalk}
 We have 
 \begin{equation}\begin{split}
  \bbE[ v(\gep)^{2} |\hat M_{\gep}(\xi)|^2]&\le C(\rho)\\
  \bbE[v(\gep)^{2}|\hat M_{\gep}(\xi+a)- \hat M_{\gep}(\xi)|^2]&\le |a|^2.
 \end{split}\end{equation}

\end{lemma}

\begin{proof}
 The proof of both bounds follows from direct computation
 We have 
 \begin{multline}
   \bbE\left[|\hat M_{\gep}(\xi)|^2\right]
   = \int_{\bbR^{2d}} \rho(x)\rho(y)e^{i\xi.(x-y)}\bbE \left[ e^{\gamma X_{\gep}(x)+\bar \gamma X_{\gep}(y)-\frac{\gamma^2}{2} K_{\gep}(x)-\frac{\bar \gamma^2}{2} K_{\gep}(y) }\right]\dd x \dd y
   \\= \int_{\bbR^{2d}} \rho(x)\rho(y)
   e^{i\xi.(x-y)} e^{|\gamma|^2 K_\gep(x,y)}\dd x \dd y
   \le \int_{\bbR^{2d}} \rho(x)\rho(y)
   e^{|\gamma|^2 K_\gep(x,y)}\dd x \dd y .
 \end{multline}
The last integral is of order $v(\gep)^2$. 
In the same fashion we have 
\begin{multline}
   \bbE[|\hat M_{\gep}(\xi+a)- \hat M_{\gep}(\xi)|^2]
   \\=\int_{\bbR^{2d}} \rho(x)\rho(y)\left(
   e^{i\xi.x}-   e^{i(\xi+a).x}  \right)\left(
   e^{-i\xi.y}-   e^{-i(\xi+a).y}  \right) e^{|\gamma|^2 K_\gep(x,y)}\dd x \dd y
   \\
   \le  |a|^2 \int_{\bbR^{2d}} \rho(x)\rho(y) |x||y|
   e^{|\gamma|^2 K_\gep(x,y)}\dd x \dd y.
\end{multline}
Since the support of $\rho$ is compact and hence bounded, the last integral is also of order $v(\gep)^2$.
\end{proof}
\noindent Now given $A$ and $d/2 < u'<u $ we define  $K_A:=K^{(1)}_A\cap K^{(2)}_A$ with  
\begin{equation}\begin{split}
K^{(1)}_A&:= \left\{  \varphi \ : \   \int_{\bbR^{d}} |\varphi(\xi)|^2(1+ |\xi|^2)^{-u'}\dd \xi\le A\right\}\\
K^{(2)}_A&:= \left\{ \varphi \ : \  \forall |a|\le 1, \  \|\varphi(\cdot+a)- \varphi \|_{\hat H^{-u}(\bbR^d)}\le A\sqrt{|a|}\right\}
\end{split}\end{equation}
It is immediate to check from Lemma \ref{criterion} that $K^{(1)}_A\cap K^{(2)}_A$ is relatively compact. To conclude the proof of Proposition \ref{tiggh} we only have to check the following

\begin{lemma}
We have 
\begin{equation}
 \lim_{A\to \infty} \sup_{\gep\in (0,\gep_0)} \bbP \left[ v(\xi) \hat M_{\gep}\notin K_A \right]=0.
\end{equation}
\end{lemma}

 \begin{proof}

 To show that 
 \begin{equation}\label{kkk1}
   \lim_{A\to \infty} \sup_{\gep\in (0,\gep_0)} \bbP \left[ v(\gep) \hat M_{\gep}\notin K^{(1)}_A \right]=0,
 \end{equation}
it is sufficient to observe that from Proposition \ref{kalk}, and the fact that $s'>d/2$, we have
\begin{equation}
  \bbE \left[ \int_{\bbR^{d}} v(\gep)^2 |\hat M_{\gep}(\xi)|^2(1+ |\xi|^2)^{-u'}\dd  \xi \right] < C'(\rho).
  \end{equation}
Then \eqref{kkk1} simply follows from Markov inequality.
Let us now prove 
 \begin{equation}\label{kkk2}
   \lim_{A\to \infty} \sup_{\gep\in (0,\gep_0)} \bbP \left[ v(\gep) \hat M_{\gep}\notin K^{(2)}_A \right]=0.
 \end{equation} 
We introduce 
\begin{equation}
\Phi(u):=\maxtwo{\xi\in \bbR^d}{|a|\le 1}\left(\frac{1+ |\xi-a|^2}{1+ |\xi|^2} \right)^{u/2}
\end{equation}
Using again Lemma \ref{kalk} and Markov inequality we have 
given $b\in \bbR^d$
\begin{equation}
\bbP \left[   \int_{\bbR^d} v(\xi)^2|\hat M_{\gep}(\xi+b)-\hat M_{\gep}(\xi)|^2(1+ |\xi|^2)^{u}\dd \xi \ge u \right] \le  \frac{C|b|^2}{u}
\end{equation}
We apply it to $b_{k,i}= 2^{-k}{\bf e}_i$ for $k\ge 1$ and $i\in \lint 1, d\rint$ where ${\bf e}_i$ are the unit coordinate vectors.
Hence setting
\begin{equation}
 K^{(3)}_A:= \left\{ \varphi  \ : \ \forall (k,i)\in \bbN \times \lint 1, d \rint, \   \|\varphi(\cdot+b_{k,i})-\varphi\|_{\hat H^{-u}(\bbR^d)} \ge \frac{ A 2^{-k/2}(\sqrt{2}-1)}{\sqrt{2}d\Phi(u)}  \right\}.
\end{equation}
We obtain after a union bound over $i$ and $k$ that 
\begin{equation}
 \bbP \left[ v(\xi)\hat M_{\gep} \notin K^{(3)}_A \right] \le  \frac{C}{A^2}
\end{equation}
Then \eqref{kkk2} follows from the inclusion $ K^{(3)}_A\subset K^{(2)}_A$, which we prove now.
Consider $\varphi\in K^{(3)}_A$
Now note that for any $\phi$ 
\begin{equation}\label{zoob}
 \max_{|a|\le 1}\frac{\|\phi(\cdot+a)\|_{\hat H^{-u}(\bbR^d)}}{ \|\phi\|_{\hat H^{-u}(\bbR^d)}}<\Phi(s).
\end{equation}
Applying this to $\phi= \varphi(\cdot +b_{i,k})-\varphi$,
we obtain that for all $(k,i)\in \bbN \times \lint 1, d \rint,$ and   $|a|\le 1$ 
\begin{equation}\label{microscoop}
  \|\varphi(\cdot +a+b_{i,k})-\varphi(\cdot +a) \|_{\hat H^{-s}(\bbR^d)}\le \frac{A 2^{-k/2}(\sqrt{2}-1)}{\sqrt{2}d} 
\end{equation}
Given $a$ with $2^{-k_0}\le |a|\le 2^{1-k_0}$, assuming without loss of generality that all coordinates are positive we can write $a$ in the following form
$$a:= \sum_{i=1}^d \sum_{k\ge k_0} \chi(k,i,a) b_{k,i},$$ 
with $ \chi(k,i,a)\in \{0,1\}$ (the decomposition is not necessarily unique). 
We write $$a_{k,i}:=  \sum_{j=1}^d \sum_{m\ge k_0} \chi(m,i,a) b_{m,j}(\ind_{\{ m\le k-1\}}+\ind_{\{m=k,j\le i\}}).$$
Then using \eqref{microscoop} and the triangle inequality we obtain that for every  $\phi \in  K^{(3)}_A$, $k\ge k_0$ and $i\in \lint 1, d \rint$
\begin{equation}
 \|\varphi(\cdot +a_{i,k})-\varphi \|_{\hat H^{-s}(\bbR^d)}\le A \left(1-\frac{1}{\sqrt{2}}\right)\sum_{m=k_0}^k2^{-k/2}\le A 2^{-k_0/2}\le A\sqrt{|a|}.
\end{equation}
Passing to the limit we obtain that $\|\varphi(\cdot +a)-\varphi \|_{\hat H^{-u}(\bbR^d)} \le A\sqrt{|a|}$, and thus that $\varphi\in K^{(2)}_A$, which concludes the proof.

\end{proof}

\bibliographystyle{plain}
\bibliography{bibliography.bib}

\end{document}